\documentclass[11pt]{amsart}

\usepackage{fullpage, amsmath, amssymb, amsthm, enumerate, url, bibentry, hyperref, color,xcolor,mathtools}
\usepackage{todonotes, float,makecell,multirow}

\usepackage[normalem]{ulem}

\newtheorem{theorem}{Theorem}
\numberwithin{theorem}{section}
\numberwithin{equation}{section}
\numberwithin{table}{section}

\newtheorem{corollary}[theorem]{Corollary}
\newtheorem{lemma}[theorem]{Lemma}

\newtheorem{proposition}[theorem]{Proposition}
\newtheorem{remark}[theorem]{Remark}
\newtheorem{conjecture}[theorem]{Conjecture}

\newcommand{\End}{\operatorname{End}}
\def\l{{\lambda}}
\def\sgn{{\rm sgn}}
\def\C{\mathbb C}
\def\F{\mathbb F}
\def\Z{{\mathbb Z}}
\def\N{{\mathbb N}}

\def\Q{{\mathbb Q}}
\def\HH{{\mathbb H}}

\def\G{\Gamma}
\def\ol{\overline}

\def\SL{{{\operatorname{SL}}}}
\def\Sym{{{\operatorname{Sym}}}}
\def\Frob{{{\operatorname{Frob}}}}

\def\End{{{\operatorname{End}}}}
\def\Span{{{\operatorname{Span}}}}

\def\mat#1#2#3#4{\left(\begin{smallmatrix}#1&#2\\#3&#4\end{smallmatrix}
	\right)}

\newcommand*\HYPERskip{&}
\catcode`,\active
\newcommand*\pFq{
\begingroup
\catcode`\,\active
\def ,{\HYPERskip}%
\doHyper
}
\catcode`\,12
\def\doHyper#1#2#3#4#5{%
\, _{#1}F_{#2}\left[\begin{matrix}#3 \smallskip \\  #4\end{matrix} \; ; \; #5\right]%
\endgroup
}

\title{Supercongruences arising from Ramanujan-Sato Series}

\author{Angelica Babei}
\address{Department of Mathematics, Howard University, Washington, D.C. 20059, USA}
\email{angelica.babei@howard.edu}

\author{Manami Roy}
\address{Department of Mathematics, Lafayette College, Easton, PA 18042, USA}
\email{royma@lafayette.edu}

\author{Holly Swisher}
\address{Department of Mathematics, Oregon State University, Corvallis, OR 97301, USA}
\email{swisherh@oregonstate.edu}

\author{Bella Tobin}
\address{Department of Mathematics, Agnes Scott College, Decatur, GA 30030 USA}
\email{btobin@agnesscott.edu}

\author{Fang-Ting Tu}
\address{Department of Mathematics, Louisiana State University, Baton Rouge, LA 70803, USA}
\email{ftu@lsu.edu}


\begin{document}

\begin{abstract}

Recently, the authors with Lea Beneish established a recipe for constructing Ramanujan-Sato series for $1/\pi$, and used this to construct 11 explicit examples of Ramanujan-Sato series arising from modular forms for arithmetic triangle groups of non-compact type. Here, we use work of Chisholm, Deines, Long, Nebe and the third author to prove a general $p$-adic supercongruence theorem through an explicit connection to CM hypergeometric elliptic curves that provides $p$-adic analogues of these Ramanujan-Sato series. We further use this theorem to construct explicit examples related to each of our explicit Ramanujan-Sato series examples.
\end{abstract}

\maketitle


\section{Introduction and Statement of Results}\label{sec:intro}

Using infinite series to approximate $\pi$ dates back to Madhava of Sangamagrama circa 1400, and was revolutionized by Ramanujan in 1914 \cite{Ramanujan} when he stated several rapidly converging series formulas for $1/\pi$ of the form
\[
\sum_{k= 0}^{\infty} \frac{(\frac{1}{2})_k (\frac{1}{d})_k (\frac{d-1}{d})_k }{k!^3} (ak+1)(\lambda_d)^k=\frac{\delta}{\pi},
\]
for $d\in \{2,3,4,6\}$, where $\l_d$ are singular values that correspond to elliptic curves with complex multiplication (CM), $a, \delta$ are explicit algebraic numbers, and $(a)_k$ denotes the rising factorial $(a)_k=a(a+1)\cdots(a+k-1)$.\footnote{We note that in fact, a similar identity was given even earlier by Bauer \cite{Bauer}.}  Ramanujan's formulas are all examples of special evaluations of hypergeometric functions. They were proved in 1987 by brothers J. and P. Borwein \cite{BB} as well as D. and G. Chudnovsky \cite{CC}, and both approaches rely on the arithmetic of elliptic integrals of the first and second kind, including the Legendre relation at singular values.  Since the 1980's series of this type have been at the forefront of algorithms to compute decimal approximations of $\pi$.   In particular, D. and G. Chudnovsky \cite{CC} established the record winning series
\begin{equation*}\label{chud}
\frac{426880\sqrt{10005}}{13591409\pi} = \sum\limits_{k=0}^\infty \left(\frac{a}{b}k+1\right) \frac{\left(\frac{1}{2}\right)_k\left(\frac{1}{6}\right)_k\left(\frac{5}{6}\right)_k}{(1)_k^3}\left( \frac{-1}{53360^3}\right)^k \!\! = \!\! \pFq{4}{3}{ \frac12 & \frac16 & \frac56&\frac{a+b}{a} }{& 1 & 1& \frac{b}{a} }{\frac{-1}{53360^3}},
\end{equation*}
where $a=545140134$, $b=13591409$.  To date the current record of computing $10^{14}$ digits of $\pi$ was set by Emma Haruka Iwao on June 8, 2022 \cite{GCloud} using \eqref{chud}.  Generalized series of this form, in which the parameters are related to modular forms, are now called Ramanujan-Sato series after Sato \cite{Sato}, whose formulation in 2002 of a new example of this type renewed interest in the area.

In 1997, Van Hamme \cite{vanHamme} developed $p$-adic analogs of several Ramanujan type series for $1/\pi$.  Analogs of this type are called Ramanujan type supercongruences, and relate truncated sums of hypergeometric series to values of the $p$-adic gamma function.  Ramanujan type supercongruences have been a rich area of study, see for example \cite{Zudilin-supercongruence, Guillera-Zudilin, LongRamakrishna, Swisher, OsburnZudilin, Beukers-supercongruence}.  In 2013, Chisholm, Deines, Long, Nebe, and the third author \cite{WIN2} prove a general $p$-adic analog of Ramanujan type supercongruences modulo $p^2$ for suitable truncated hypergeometric series arising from CM elliptic curves.

Recently, the authors with Lea Beneish \cite{WIN5} established a recipe for constructing Ramanujan-Sato series for $1/\pi$, and used this to construct 11 explicit examples of Ramanujan-Sato series arising from modular forms for arithmetic triangle groups of non-compact type. In this article, we establish supercongruences related to Ramanujan-Sato series given in \cite{WIN5} motivated by the work of Chisholm et al. \cite{WIN2}.

For series of $1/\pi$ that arise via modular forms for a group $\Gamma$ such as in \cite{CCL, ChenGlebovGoenka, WIN5}, there is a connection between the series and known $p$-adic analogs through families of elliptic curves whose monodromy group is $\Gamma$. Inspired by Chisholm et al. \cite{WIN2}, we make this connection explicit through hypergeometric evaluations. In fact, it motivated us to revisit \cite[Theorem 2]{WIN2} using arithmetic properties of elliptic families.

\subsection{Notations for main results}
We use the following notation for the hypergeometric functions appearing in this work.  Given $\alpha \in \C$, $n\in \N$, ${\bf b}= (b_1, \ldots, b_n) \in (\C^\times)^n$, define
\begin{equation}\label{def:F_a}
F_{\alpha,{\bf b}}(z) = \pFq{n+1}{n}{ \frac{1+\alpha}{\alpha} & b_1 & b_2 & \cdots & b_n}{ & \frac{1}{\alpha} & 1 & \cdots & 1}{z} = \sum_{k\geq0} (\alpha k+1) \frac{(b_1)_k\cdots(b_n)_k}{(1)_k^n} z^k,
\end{equation}
where the $_{n+1}F_n$ is taken to be the limit as $\alpha\rightarrow 0$ when $\alpha=0$ (yielding an $_nF_{n-1}$). 
We further write $[F_{\alpha,{\bf b}}]_N(z)$ to denote the truncation of the series \eqref{def:F_a} at $z^N$.

When we evaluate this truncation $[F_{\alpha, \bf{b}}]_N$ for some $\lambda$, i.e., the function $F$ written as a series, truncated at the $z^N$ term, and then evaluated at $\lambda$, we write   
\[ [F_{\alpha, \bf{b}}]_N(\lambda)= \sum\limits_{k= 0}^N(\alpha k+1)\frac{(b_1)_k\cdots (b_n)_k}{(1)_k^n}\lambda^k.\]

For $d\in \{2,3,4,6\}$ let $\widetilde E_d(t)$ denote the following families of elliptic curves\footnote{We note that our definition for $\widetilde{E}_2(t)$ differs by a sign from that in \cite{WIN2}.  This choice makes some of our results cleaner and more consistent with other $d$ values, and does not affect the results in \cite{WIN2}.} parameterized by $t$,
\begin{align*}
\widetilde{E}_2(t): & \;\; y^2=x(1-x)(x-t), \\
\widetilde{E}_3(t):  & \;\;y^2+xy+\frac{t}{27}y=x^3,   \\
\widetilde{E}_4(t): & \;\;y^2=x(x^2+x+\frac{t}{4}), \\
\widetilde{E}_6(t): & \;\;y^2+xy=x^3-\frac{t}{432}.
\end{align*}

We note that with this notation we have that $\pi\cdot F_{0,(\frac{1}{d}, \frac{d-1}{d})}(t)$ is a fundamental period for $\widetilde E_d(t)$ for $d\in \{2,3,4,6\}$.

A curve $E$ over a number field $K$ has CM if its geometric endomorphism ring End$(E(\overline{\Q}))$ is an order of an imaginary quadratic field.  For $t$ such that $\widetilde E_d(t)$ has CM, set 
$\l = -4t(t-1)$ and assume $|\l|<1$ for any embedding.  Then as discussed in Chisholm et al. \cite{WIN2}, there exists unique algebraic numbers $\alpha,\delta$ depending on $d$ and $\l$ such that
\begin{equation}\label{eq:ram}
\sum_{k\geq 0} (\alpha k+1) \frac{(\frac12)_k(\frac1d)_k(\frac{d-1}d)_k}{k!^3}\l^k = \frac{\delta}{\pi}.
\end{equation}
Moreover, specific CM values of $\l$ as well as the corresponding $\alpha,\delta$ can be computed from data of the Borweins \cite{BB}.  Using the notation introduced above, \eqref{eq:ram} can also be rewritten as the following product (see Lemma \ref{lem:4F3prod}),
\[
F_{0,(\frac{1}{d}, \frac{d-1}{d})}(t)\cdot F_{a_2,(\frac{1}{d}, \frac{d-1}{d})}(t) = \frac{\delta}{\pi},
\] 
where $t=\frac{1-\sqrt{1-\l}}2$, $a_2=\frac{2\alpha(1-t)}{1-2t}$, and $\delta$ is a unique algebraic number depending on $t$.

Additionally, the Ramanujan-Sato series constructed in \cite{WIN5} can be written in the form 
\begin{equation}\label{eqn:examples}
\frac{\delta}{\pi} = \sum_{k\geq 0} \left(\alpha k+1\right) A_k \lambda^k = F_{0,{\bf b}}(\lambda) \cdot F_{2\alpha,{\bf b}}(\lambda),
\end{equation}
for parameters $\delta, \alpha, {\bf b}, \lambda$, where $A_k$ is defined by $\sum_{k\geq 0} A_k z^k=F_{0,{\bf b}}(z)^2$ (see \eqref{eq:sqcoeffs}). The 11 explicit examples of Ramanujan-Sato series given in \cite{WIN5}, rewritten from this perspective,
are described by the data in Table~\ref{tab:RSData}, which uses some notation from Corollary~\ref{cor:WIN5} below. 

\begingroup
\renewcommand{\arraystretch}{1.6}
\begin{table}[H]
\label{tab:RSData}

\begin{tabular}{|c|c|c|c|c|c|c|c|}
\hline
Ex. in \cite{WIN5} & $\delta$ & $N$ & $\alpha_N$ & ${\bf b}$ & $\lambda$& $\mu$ & $d$ \\
\hline
4.2 & $2 \sqrt{3}$ & $3$ &$3\sqrt{2} + 4$ & $\frac14, \frac34$ & $\frac{1}{6}(3-2\sqrt{2})$ &$\frac{1}{6}(3-2\sqrt{2})$ &$4$ \\
\hline
4.3 & $\frac{9}{4}\sqrt{2}$ & $5$ & $\frac{1}{4}(20+9\sqrt{5})$ & $\frac14, \frac34$ & $\frac{1}{18}(9-4\sqrt{5})$&$\frac{1}{18}(9-4\sqrt{5})$ &$4$ \\
\hline
4.4 & $\sqrt{3}$ & $3$ & $4$ & $\frac14, \frac34$ & $-\frac13$ & $-\frac13$ & $4$ \\
\hline
4.5 & $\frac83$ & $5$ & $\frac13 (10+4\sqrt{5})$ & $\frac14, \frac34$ & $\frac14 (2-\sqrt{5})$& $\frac14 (2-\sqrt{5})$& $4$ \\
\hline
4.6 & $3\sqrt{3}$ & $2$& $3(\sqrt{2}+1)$ & $\frac13, \frac23$ & $\frac14(2-\sqrt{2})$ & $\frac14(2-\sqrt{2})$ &$3$ \\
\hline
4.7 & $2\sqrt{2}+3$ & $2$ & $4+2\sqrt{2}$  & $\frac12, \frac12$ & $3-2\sqrt{2}$ &  $3-2\sqrt{2}$ &  $2$ \\
\hline
5.2 & $\frac{5^2\Gamma(\frac18)^2\Gamma(\frac38)^2}{2^{9/2}\cdot 3 \pi^3}$ & $2$ & $\frac{14}{3}$ & $\frac12, \frac16, \frac56$ & $\frac{27}{125}$ & $\frac 12 -\frac{7}{5\sqrt{10}}$ & $6$\\
\hline
5.3 & $\frac{5^2\sqrt{3}\Gamma(\frac13)^6}{2^{17/3} \pi^4}$ & $3$ & $\frac{11}{2}$ & $\frac12, \frac16, \frac56$ & $\frac{4}{125}$ & $\frac 12 -\frac{11}{10\sqrt{5}}$& $6$\\
\hline
5.4 & $\frac{\Gamma(\frac{1}{24})\Gamma(\frac{5}{24})\Gamma(\frac{7}{24})\Gamma(\frac{11}{24})}{2^{5/2} \pi^3}$ & $3$ & $4$ & $\frac12, \frac14, \frac34$ & $\frac19$ & $\frac 12 -\frac{\sqrt{2}}{3}$ & $4$ \\
\hline
5.5 & $\frac{3^2 \Gamma(\frac{1}{24})\Gamma(\frac{5}{24})\Gamma(\frac{7}{24})\Gamma(\frac{11}{24})}{2^5 \pi^3}$ & $2$ & $3$ & $\frac12, \frac13, \frac23$ & $\frac12$ &$\frac 12 -\frac1{2\sqrt{2}}$ & $3$ \\
\hline
5.6 & $\frac{3^2\cdot 5 \Gamma(\frac{1}{15})\Gamma(\frac{2}{15})\Gamma(\frac{4}{15})\Gamma(\frac{8}{15})}{2^8 \pi^3}$ & $5$ & $\frac{33}{8}$ & $\frac12, \frac13, \frac23$ & $\frac{4}{125}$ & $\frac 12 -\frac{11}{10\sqrt{5}}$ & $3$\\
\hline
\end{tabular}
\caption{Ramanujan-Sato Example Data}
\end{table}
\endgroup

\begin{corollary}[Corollary of {\cite[Theorem 1.1]{WIN5}}]\label{cor:WIN5} Let $\Gamma$ be a discrete subgroup of $\SL_2(\mathbb{R})$ commensurable with $\SL_2(\Z)$ such that $\left(\begin{smallmatrix}1&1\\0&1\end{smallmatrix}\right) \in \Gamma$, and let $X(\tau)$ be a Hauptmodul of $\Gamma$.  Let $Z(\tau)$ be a weight-$k$ modular form for $\Gamma$ such that $\frac{1}{2\pi i}\frac{dX}{d\tau} = U(\tau)X(\tau)Z(\tau)$ and when $\tau\in D$, a domain of $\mathbb{H}$, $Z(\tau)= F_{0,{\bf b}}(X(\tau))^2$ for ${\bf b}\in \Q^n$.  Further assume there exist $\varepsilon \in \C$ and $\gamma = \left(\begin{smallmatrix}a&b\\c&-a\end{smallmatrix}\right) \in \SL_2(\mathbb{R})$ such that 
\[
\left( Z \vert_k  \gamma \right) (\tau) = \varepsilon Z (\tau).
\]
Set $M_N(\tau):=Z(\tau)/Z(N\tau)$ for $N\in \N$ satisfying $\frac ac(1-N)\in \Z$, and let $\tau_0=\frac{a}{c}+ \frac{i}{c\sqrt N}$. Set $v(\tau_0)$ as the special value $v(\tau_0)=U(\tau_0)X(\tau_0) \left(\frac{dM_N}{dX}\right) \left|_{X=X(\tau_0)} \right.$.  

Then if $\gamma\tau_0 =\frac ac + \frac{i\sqrt N}{c} \in D$, we have\footnote{For the case where $\tau_0 \in D$ a similar result can be stated, but we omit this here.}
\[
\frac{\delta_N}{\pi} = F_{0,{\bf b}}(X(\gamma\tau_0)) \cdot F_{2\alpha_N,{\bf b}}(X(\gamma\tau_0)),
\]
where $\delta_N = \frac{ck\sqrt{N}}{2v(\tau_0) }$, and $\alpha_N = \frac{2NU(\gamma\tau_0)}{v(\tau_0)}$.
\end{corollary}

\begin{remark} \label{rmk: forms}
We note that the form given in \cite[Ex. 4.2 - 4.6]{WIN5} is different than \eqref{eqn:examples}, but using Corollary \ref{cor:WIN5} we can rewrite them in the form of \eqref{eqn:examples}.  For instance for \cite[Ex. 4.2]{WIN5} we do this by setting 
$Z(\tau) = F_{0,{\bf b}}(t_2(\tau)/(t_2(\tau)-1))^2$ with ${\bf b} = (1/4,3/4)$ and $X(\tau) = \frac{t_2(\tau)}{t_2(\tau)-1}$, using the notation in \cite{WIN5}, which yields that $U(\tau) = \frac{1}{1-t_2(\tau)} = 1-X(\tau)$ and
\[
\frac{dM_N}{dX} = N\left(\frac{d\left(\frac{dX}{dY}\right)}{dX}\cdot \frac{Y(1-Y)}{X(1-X)} + \frac{(1-2Y)}{X(1-X)} - \frac{dX}{dY}\cdot\frac{Y(1-Y)(1-2X)}{X^2(1-X)^2} \right). 
\]
This works similarly for \cite[Ex.~4.6]{WIN5} with ${\bf b} = (1/3,2/3)$ and $t_3(\tau)$ in the place of $t_2(\tau).$ The form given in \cite[Ex.~4.7]{WIN5} matches \eqref{eqn:examples}, so we apply Corollary \ref{cor:WIN5} with $X(\tau)=t_{\infty}(\tau)$ and $U(\tau)=1-X(\tau)$.
\end{remark}

\subsection{Statement of main results}
We can now state our main theorem.

\begin{theorem} \label{thm:main}
Let $d\in \{2,3,4,6\}$ and set ${\bf b} = (\frac{1}{d}, \frac{d-1}{d})$, or ${\bf b} = (\frac12, \frac{1}{d}, \frac{d-1}{d})$.  Suppose $X(\tau)$ and $Z(\tau)= F_{0,{\bf b}}(X(\tau))^2$ satisfy all of the conditions in Corollary \ref{cor:WIN5} with $\gamma=\left(\begin{smallmatrix}a&b\\c&-a\end{smallmatrix}\right)$ and $\tau_0=\frac{a}{c}+ \frac{i}{c\sqrt N}$, and furthermore if $\l = X(\gamma\tau_0)$, then $F_{0,{\bf b}}(\l) = F_{0,(\frac{1}{d}, \frac{d-1}{d})}(\mu)^r$ for some $r$, where $\mu=\l$ when ${\bf b} = (\frac{1}{d}, \frac{d-1}{d})$, and $\mu=\frac{1 \pm \sqrt{1-\l}}{2}$ when ${\bf b} = (\frac12, \frac{1}{d}, \frac{d-1}{d})$.

Assume that $\Q(\l)$ is a totally real field, $|\l|, |\mu|<1$, and $\widetilde E_d(\mu)$ has CM (equivalently $\tau_0$ is a CM point).  Let $p$  
be a prime such that $X(\tau_0), \lambda$, $\alpha_N$ can be embedded into $\Z_p^\times$, and $p$ is unramified in $\Q(\mu)$.  Suppose that $\widetilde E_d(\mu)$ has good reduction modulo $p$.

Then, it follows that if ${\bf b} = (\frac{1}{d}, \frac{d-1}{d})$, 
\[
[F_{0,(\frac{1}{d}, \frac{d-1}{d})}\cdot F_{2\alpha_N,(\frac{1}{d}, \frac{d-1}{d})}]_{p-1}(X(\gamma\tau_0)) \equiv 
\begin{cases}
p & \text{ if } \left ( \frac{-c^2 N}{p} \right ) =1 \\
-p & \text{ otherwise }
\end{cases}
\pmod{p^2},
\]
if ${\bf b} = (\frac12, \frac{1}{d}, \frac{d-1}{d})$,
\begin{multline*}
[F_{0,(\frac12, \frac{1}{d}, \frac{d-1}{d})}\cdot F_{2\alpha_N,(\frac12, \frac{1}{d}, \frac{d-1}{d})}]_{p-1}(X(\gamma\tau_0)) \equiv \\
\begin{cases}
p u_p^2 & \text{ if } \left ( \frac{-c^2 N}{p} \right ) =1 \text{ and }\left(\frac{1-\l}p \right)=1 \\
-\left( \frac{k_d}{p} \right) p u_p^2 & \text{ if } \left ( \frac{-c^2 N}{p} \right ) =1 \text{ and }\left(\frac{1-\l}p \right)=-1 \\
0 & \text{ otherwise }
\end{cases}
\pmod{p^2},
\end{multline*}
where $u_p$ is a fixed $p$-adic unit root of the geometric Frobenius at $p$ acting on the first cohomology of the elliptic curve $\widetilde E_d(\mu)$, and $k_2=-1$, $k_3=-3$, $k_4=-2$. 
\end{theorem} 

\begin{remark}
\label{rmk:conj}
We note that work of Beukers (see Theorem \ref{thm:Beukers}) as well as computational observations indicate that the series obtained from Theorem \ref{thm:main} holds modulo $p^3$ in some cases but not for others. See Section \ref{sec: modp3} 
for further discussion. 
\end{remark}

In addition to Theorem \ref{thm:main}, which is a general statement, in Section \ref{sec:examples} we provide we provide specific examples using the Ramanujan-Sato series given in \cite{WIN5} (as in Table~\ref{tab:RSData}).

Hypergeometric functions, such as those appearing in Ramanujan-Sato series as well as their $p$-adic truncations and also those over finite fields, encode abundant information on the arithmetic invariants of hypergeometric motives such as their periods, Galois representations, and the $p$-adic unit roots of the local $L$-functions due to Dwork's unit roots theory. They are also closely connected to modular forms through the modularity theorem and the differential equations satisfied by certain modular forms (see \cite{ChenGlebov, WIN3+, RRV, Yang-Schwarzian, ChenGlebovGoenka}, for example).

Using Theorem \ref{thm:main}, we prove a few results in Section \ref{sec: modular} connecting the hypergeometric functions appearing in the series in \cite[Ex. 5.2 - 5.6]{WIN5} with Galois representations and special $L$-values of modular forms of weight $3$. For instance, specific examples of Corollary~\ref{cor: modularity} give us the following result.

\begin{corollary}
\label{cor:ap-lmfdb-examples}
For primes $p\geq7$, 
    \begin{align*}
       \left[ F_{0,(\frac12, \frac{1}{6}, \frac{5}{6})}\right]_{p-1}\left(\frac{27}{125}\right)& \equiv a_p( f_{\href{https://www.lmfdb.org/ModularForm/GL2/Q/holomorphic/800/3/g/a/}{800.3.g.a}})  \pmod{p^2} \\
         \left[ F_{0,(\frac12, \frac{1}{6}, \frac{5}{6})}\right]_{p-1}\left(\frac{4}{125}\right)&\equiv a_p( f_{\href{https://www.lmfdb.org/ModularForm/GL2/Q/holomorphic/300/3/g/b/}{300.3.g.b}})  \pmod{p^2}\\
        \left[ F_{0,(\frac12, \frac{1}{4}, \frac{3}{4})}\right]_{p-1}\left(\frac{1}{9}\right)&\equiv a_p( f_{\href{https://www.lmfdb.org/ModularForm/GL2/Q/holomorphic/24/3/h/a/}{24.3.h.a}})  \pmod{p^2}\\
         \left[F_{0,(\frac12, \frac{1}{3}, \frac{2}{3})}\right]_{p-1}\left(\frac{1}{2}\right) &\equiv a_p( f_{\href{https://www.lmfdb.org/ModularForm/GL2/Q/holomorphic/24/3/h/b/}{24.3.h.b}})  \pmod{p^2}\\
         \left[F_{0,(\frac12, \frac{1}{3}, \frac{2}{3})}\right]_{p-1}\left(\frac{4}{125}\right)&\equiv a_p( f_{\href{https://www.lmfdb.org/ModularForm/GL2/Q/holomorphic/15/3/d/b/}{15.3.d.b}})  \pmod{p^2},
    \end{align*}
    where the subscripts on the modular forms $f$ above are their LMFDB labels \cite{LMFDB} and $a_p(f)$ is the $p$th Fourier coefficient of $f$. 
\end{corollary}

Corollary~\ref{cor: modularity}, which is proved in Section \ref{sec: modular}, uses modularity of CM elliptic curves and $L$-functions to give values of truncated hypergeometric functions in terms of Fourier coefficients of modular forms. We also obtain in \S \ref{subseq: specialvals} values of the corresponding untruncated hypergeometric functions both in terms of periods $\Omega_K$, and conjecturally special $L$-values of these same modular forms.

\subsection{Outline of the paper}
The rest of this paper is organized as follows.  In Section \ref{sec:ECMC}, we describe the connection between the hypergeometric evaluations for $1/\pi$ given by Beneish and the authors in \cite{WIN5} and their $p$-adic counterparts provided in Theorem \ref{thm:main} through the study of hypergeometric elliptic curves and modular forms associated to certain triangle groups. We also provide several lemmas needed for the proof of Theorem \ref{thm:main}.  In Section \ref{sec:proof}, we prove Theorem \ref{thm:main}.  In Section \ref{sec:examples}, we give 11 examples of Theorem \ref{thm:main}, which stem from applications of \cite[Thm. 2]{WIN5}, as well as 4 different examples of supercongruences obtained by directly applying Theorem \ref{thm:WIN2a} to the series in \cite[Ex. 4.2 - 4.5]{WIN5}. In Section \ref{sec: modular} we describe connections between hypergeometric functions of the form $F_{0,(\frac12, \frac{1}{d}, \frac{d-1}{d})}$ and weight $3$ modular forms, modular Galois representations, and special $L$-values of weight $3$ modular forms. We further discuss modulo $p^3$ supercongruence conjectures and results in Section \ref{sec: modp3}.  In the Section \ref{sec:appendix} Appendix, we provide an explanation of Theorem \ref{thm:WIN2b} through the arithmetic properties of hypergeometric elliptic families, and further record some parallel results in terms of finite hypergeometric functions introduced by Beukers, Cohen, and Mellit \cite{BCM}, which allow us to see the supercongrunce in Theorem \ref{thm:WIN2b} from a character sum perspective.

\subsection{Acknowledgements}

This material is based upon work supported by the National Science Foundation grant DMS-1928930 while the authors were in residence at the Simons Laufer Mathematical Sciences Research Institute in Berkeley, California, during the summer of 2023. The third author is further supported by NSF grant DMS-2101906, the fifth author by NSF grant DMS-230253, and the fourth author by an AMS-Simons travel grant.  The fifth author was hosted at Oregon State University during the fall of 2023 while working on this project. The authors would also like to thank Michael Allen, Frits Beukers, and Ling Long for their helpful discussions.  The authors would like to thank the referee for valuable suggestions and comments.

\section{Hypergeometric Elliptic Curves and Triangle Modular Curves} 
\label{sec:ECMC}

The  connection between the hypergeometric evaluations for $1/\pi$ in \cite{WIN5} and their $p$-adic version in this paper arises from the moduli interpretation of such hypergeometric elliptic curves.
 
In particular, a group $\G \in \{\Gamma(2), \Gamma_0(3), \Gamma_0(2)\}$ is a monodromy group for a family of elliptic curves as given in Table \ref{fams}, where  $t$ is the Hauptmodul of $\G$ taking values at $0$, $1$ at two of the cusps and having a pole at the elliptic point of the other cusp, and $\widetilde{E}_d(t)$ is a family of  elliptic curves over points in $X_\G$ for $t \ne 0, 1, \infty$. In Table \ref{fams} we also provide the element $W$ lying in the normalizer of the group $\G$ that switches the elements $t$ and $1-t$. 
\begingroup
\renewcommand{\arraystretch}{1.75}
\begin{table}[h!]

\caption{Families of hypergeometric elliptic curves }
\label{fams}
\begin{tabular}{|c|c|c|c|c|c|}\hline
${\widetilde{E}}_d(t)$& ${\bf b}$&  $\G$& $W$\\
\hline
 $\widetilde{E}_2(t):  y^2=x(1-x)(x-t)$ & $(\frac12,\frac12)$ &  $\G(2)$ & $\mat 0{-1}10$\\ \hline
   $ \widetilde{E}_3(t): y^2+xy+\frac{t}{27}y=x^3$ & $(\frac13,\frac23)$ &   $\G_0(3)$  & $\frac1{\sqrt 3}\mat 0{-1}30$ \\ \hline
 $\widetilde{E}_4(t): y^2=x(x^2+x+\frac t4)$ & $(\frac14,\frac34 )$  & $\G_0(2)$&  $\frac1{\sqrt 2}\mat 0{-1}20$\\ \hline
   $ \widetilde{E}_6(t): y^2+xy=x^3-\frac{t}{432}$& $(\frac16,\frac56 )$ & -- &  --\\ \hline 
\end{tabular}
\end{table}
\endgroup
For the $d=6$ case, the function $t$ 
is related to the modular group $\mbox{PSL}_2(\Z)$ by $4t(1-t)=1728/j$, 
where $j$ is the usual $j$-invariant, the Hauptmodul on  $\mbox{PSL}_2(\Z)$ with $j(i)=1728$, $j(\zeta_3)=0$, and $j(i\infty)=\infty$.

\subsection{Hypergeometric CM Elliptic Curves}\label{CMcurves} 

The main result of Chisholm, et al. in \cite{WIN2} is as follows. 

\begin{theorem}[Chisholm, et al. {\cite[Theorem 1]{WIN2}}] \label{thm:WIN2a}
For $d\in\{2,3,4,6\}$, let $\l\in \overline \Q$ such that  $\Q(\l)$ is totally real,
the elliptic curve  $\widetilde{E}_d(\frac{1\pm\sqrt{1-\lambda}}{2})$ has CM, and $|\l|<1$ for an embedding of $\lambda$ to $\C$.  For each prime $p$ that is unramified  in $\Q(\sqrt{1-\l)}$ and coprime to the discriminant of $\widetilde{E}_d(\frac{1-\sqrt{1-\lambda}}{2})$ such that 
$\alpha$, $\l$ give a series of the form \eqref{eq:ram} and can be embedded in $\Z_p^\times$ 
(and we fix such embeddings), then
\[
[F_{\alpha,\left(\frac{1}{2},\frac{1}{d},\frac{d-1}{d}\right)}]_{p-1}(\lambda)\equiv  {\sgn} \cdot \left( \frac{1-\l}p \right) \cdot p \pmod {p^2},
\]
where $\left(\frac{1-\l}p \right)$ is interpreted as a Legendre symbol over $\Z_p$, which is possible since $\l$ can be embedded in $\Z_p$, and $\sgn=\pm 1$, equaling $1$ (or $-1$) if and only if  $\widetilde{E}_d(\frac{1\pm\sqrt{1-\lambda}}{2})$ is ordinary (or supersingular) at $p$.
\end{theorem} 

As discussed in \cite[Remarks 1 and 2]{WIN2}, this congruence also holds when $|\l|\geq 1$ and has been conjectured by Zudilin to hold modulo $p^3$ \cite{ Guillera-Zudilin, GuoZudilin, Zudilin-supercongruence}. See, for example, \cite{Beukers-supercongruence, Mortenson-padic,  Zudilin-supercongruence} for some confirmed cases.  The following is a companion to Theorem \ref{thm:WIN2a}.

\begin{theorem}(Revision of \cite[Theorem 2]{WIN2})\label{thm:WIN2b}
Assume the notation and assumptions in Theorem \ref{thm:WIN2a}.  Then,
\[
[F_{0,(\frac12, \frac{1}{d}, \frac{d-1}{d})}]_{p-1}(\l) \equiv L \pmod {p^2},
\]

where 
\[
L= \begin{cases}
0 & \text{if } \widetilde{E}_d(\frac{1-\sqrt{1-\lambda}}{2}) \text{ is supersingular at } p\\
u_p^2  & \text{if } \widetilde{E}_d(\frac{1-\sqrt{1-\lambda}}{2}) \text{ is ordinary at } p \text{ and }\left(\frac{1-\l}p \right)=1, \\

\left( \frac{k_d}{p} \right)u_p^2  & \text{if } \widetilde{E}_d(\frac{1-\sqrt{1-\lambda}}{2}) \text{ is ordinary at } p \text{ and }\left(\frac{1-\l}p \right)=-1, 
\end{cases}
   \]
and in the ordinary case, $u_p$ is a fixed $p$-adic unit root of the geometric Frobenius at $p$ acting on the first cohomology of the elliptic curve $\widetilde{E}_d(\frac{1-\sqrt{1-\lambda}}{2})$ and $k_2=k_6=-1$, $k_3=-3$, $k_4=-2$.  
\end{theorem}

We note that the statement of \cite[Thm. 2]{WIN2} is missing the case when $\left(\frac{1-\l}p \right)=-1$.  A proof of this missing case is supplied in the Appendix \S \ref{sec: WIN2proof}.  We further note that computationally Theorem~\ref{thm:WIN2b} appears to hold modulo $p^3$ for ordinary primes, however for supersingular primes the result only holds modulo $p^2$ generically. This is discussed further in Section \ref{sec: modp3}.

The proofs of Theorems \ref{thm:WIN2a} and \ref{thm:WIN2b} given by Chisholm, et al. in \cite{WIN2} utilize the connection between periods of elliptic curves and formulas for $1/\pi$ and their $p$-adic versions.  We summarize these ideas here.

Recall that the de Rham space $H_{DR}^1(\widetilde{E}_d(t)/ \C[[t]])$ is 2-dimensional, where one of its spanning vectors is given by the class $\omega_t$ of the standard invariant holomorphic differential 1-form (modulo exact forms). Another spanning vector can be obtained via the Gauss-Manin connection. For the families of elliptic curves $\widetilde{E}_d(t)$, such a vector is given by $\partial_t \omega_t$, where $\partial_t$ denotes the partial derivative with respect to $t$ \cite{Beukers-Stienstra, Katz70, Katz72, Katz, Manin, Yoshida-DE}. 

On the other hand,  if  $\widetilde{E}_d(t)$ has CM, then its endomorphism ring $R$ corresponds to an order in an imaginary quadratic field $K=\Q(\sqrt{-D})$ with $-D$ a fundamental discriminant. Moreover, $R$ induces an action on $H_{DR}^1(\widetilde{E}_d(t); \C[[t]])$ for which $\omega_t$ is a common  eigenvector. By Chowla and Selberg \cite{ChowlaSelberg},  for any $C \in H_1(\widetilde{E}_d(t), \Z)$, the period $\int_C\omega_t$ satisfies 
\begin{equation}
\label{period}
\int_C \omega_t \sim  \pi \Omega_K,
\end{equation}
where $\sim$ denotes equivalence up to multiplication by an algebraic number. Here, $\Omega_K \in \C^\times$ is uniquely determined by $K$ provided by the so-called Chowla-Selberg formula (see for example \cite{ChowlaSelberg} or \cite[\S 6.3]{BGHZ}), 
\begin{equation}
\label{omegadef}
\Omega_K :=  \frac{1}{\sqrt{\pi}}  \prod_{j=1}^{D-1} \Gamma \left( \frac{j}{D}\right)^{\chi(j) \cdot \frac{n}{4h}},
\end{equation}
where $n$ is the order of the unit group in $K$, $\chi$ is the quadratic Dirichlet character modulo $D$ for $K$, and $h$ is the class number of $K$.

There is another common eigenvector $\nu_t$ for the action of $R$ on $$H_{DR}^1(\widetilde{E}_d(t); \C[[t]]) = \Span_\C(\{[\omega_t], [\partial_t \omega_t]\}),$$ which is coming from a differential of the second kind.  From the Legendre relation as well as work of both Masser and Chowla--Selberg \cite{ChowlaSelberg}, 
\begin{equation}
\label{CS}
    \int_C \omega_t \cdot \int_C \nu_t \sim \pi,
\end{equation}  and  there is a unique algebraic number $c_t$  such that $\nu_t=\omega_t + c_t \cdot t  \partial_t \omega_t$.

 The $p$-adic analog of \eqref{CS} follows from the formal group of the elliptic curve.  See \cite[Ch. IV]{Silverman} for background on formal groups, and \cite{LL, Beukers-Stienstra} for details on the role formal group laws play in the existence of congruences.  Expand
\begin{equation}
  \label{def:wv}
   \omega_t=\sum_{n\ge 0} a_t(n)\xi^n d\xi \qquad \text{ and } \qquad  \nu_t=\sum_{n\ge 0} b_t(n)\xi^n d\xi,
\end{equation}
where $\xi=-\frac{x}{y}$ is a local uniformizer of $\widetilde{E}_d(t)$ at infinity over $\Z_p$. For a fixed algebraic integer $t$ of degree $n$ over $\F_p$, we denote by  
\begin{equation}
\label{def:A}
    A := W(\F_p(t)),
\end{equation} the ring of Witt vectors of $\F_p(t)$, which is isomorphic to $\Z_p[\zeta_{p^n-1}]$, the ring of integers of the unique unramified extension of degree $n$ over $\Q_p$. For the cases explored in this paper, we either have $A=\Z_p$, or $A$ is the maximal order in the quadratic unramified extension of $\Q_p$.

When the parameter $t$ is specified at values $\mu$, the elliptic curve $\widetilde{E}_d(\mu)$ has eigenvectors for the action of $R$ given by $\omega := \omega_\mu $ and $\nu := \nu_\mu$.  In this setting we drop the subscripts to simplify the notation. Moreover, via the isomorphism in \cite[Lem. 5.1.2]{Katz-crystalline}, we can identify $\omega$ and $\nu$ via their formal integrals by 
\begin{equation}
  \label{def:wvInt}
   \omega=\sum_{n\ge 1} \frac{a(n-1)}{n}\xi^n \qquad \text{ and } \qquad   \nu=\sum_{n\ge 1} \frac{b(n-1)}{n}\xi^n.
\end{equation} Following \cite[Prop. 16]{WIN2} and \cite[Thm. 6.1]{Katz-crystalline}, we obtain the following $p$-adic version of \eqref{CS}.

\begin{lemma}\label{lem:WIN2Prop16}
Suppose ${\bf b}=(\frac1d, \frac{d-1}{d})$, and $\widetilde{E}_d(\lambda)$ satisfies the conditions of Theorem \ref{thm:main}.  Then,
\[
a(p-1)b(p-1) \equiv \begin{cases} p &\text{ if } \widetilde{E}_d(\lambda) \text{ is ordinary at }p \\ -p &\text{ if } \widetilde{E}_d(\lambda) \text{ is supersingular at }p \end{cases} \pmod{p^2}. 
\]
\end{lemma}

\begin{proof}
We observe that using the notation of Chisholm et al. \cite{WIN2}, $A=\Z_p$, $\omega, \nu$ are eigenfunctions of the endomorphism ring $R$ of the elliptic curve, 
and from the proof of \cite[Thms.1,2]{WIN2}, there exists a degree $p$ (resp. $p^2$) Frobenius lifting $\Phi$ that commutes with the induced action of $R$ on the deRham coholomogy in the ordinary (resp. supersingular) case.  Thus, we can apply Katz \cite[Thm. 6.1]{Katz-crystalline} to obtain 
\[
a(p^2-1) - \text{tr}(\Phi) a(p-1) + p \equiv 0 \pmod{p^2},
\]
 
Here this corresponds with \cite[(5.1)]{WIN2} with $\sigma$ the identity map, $r=2$, $n=1$ and furthermore, $\omega$ is annihilated  by $\Phi^2-\text{tr}(\Phi)\Phi+p$.  As this is the same setting as \cite[Prop. 16]{WIN2}, this time $\widetilde{E}_d(\lambda)$ being defined over $\Q(\lambda)$ and $p>3$ such that $\Q_p(\lambda) \cong \Q_p$,   following the proof directly yields that
\[
a(p-1)b(p-1) \equiv \begin{cases} p &\text{ if } \widetilde{E}_d(\lambda) \text{ is ordinary at }p \\ -p &\text{ if } \widetilde{E}_d(\lambda) \text{ is supersingular at }p \end{cases} \pmod{p^2}.
\]
 \end{proof}

\subsection{Implications of modularity}
\label{sec:modform}

 The modularity of the functions $Z(\tau)$ in Corollary \ref{cor:WIN5} allows us to obtain information related to the elliptic curves  in Table~\ref{fams}. 
 In \eqref{CS}, we write $\pi$ in terms of periods of elliptic curves. We can make  this connection explicit via modular forms expressed as hypergeometric  series.  The Picard-Fuchs equation of $\widetilde{E}_d(t)$ has holomorphic solution around $t=0$  given  by $F(t) := F_{0,(\frac1d, \frac{d-1}{d})}(t)$ (see \cite{AAR}). Suppose we choose a Hauptmodul $t(\tau)$ of $\Gamma$ such that $F(t(\tau))$ is a weight $1$ modular form, and consider the elliptic curve $\widetilde{E}_d(t(z_0))$  for a fixed CM point $z_0 \in \HH$. It follows from \eqref{period} and the definition of $\Omega_K$ that we can choose a suitable 1-cycle (path) $C$ such that given $\omega$, i.e. the standard invariant holomorphic differential $1$-form on $\widetilde{E}_d(t(z_0))$, then 
$$
  \int_C \omega \sim \pi\cdot F(t(z_0)),
$$
where $\sim$ denotes equivalence up to multiplication by an algebraic number.  
One can restate \eqref{CS} by considering the quasi-period 
 $$
   \int_C  v=\left(B\int_C \omega + t\partial_t \int_C \omega\right) \bigg |_{t=t(z_0)} 
$$
for a specific algebraic number $B\in \overline{\Q}$. 
Thus with $z_0$ as above such that $F(t(z_0))\ne 0$, there exists a unique algebraic number $B(z_0)\in \overline{\Q}$ such that
\begin{equation}
\label{eq:givesB}
F(t(z_0))\left(t\frac{dF}{dt}|_{t=t(z_0)} + B(z_0) F(t(z_0))\right) \sim \frac{1}{\pi}.  
\end{equation}
 From the proof of \cite[Lemma 9]{WIN2}, the specific value of $B(z_0)$ is given by
\begin{equation}\label{eqn:constant c}
B(z_0)=-t(z_0)\frac{d \tau}{d t}\left(\frac 1{F(t(z_0))} \frac{dF}{dt}\bigg|_{t=t(z_0)} - \frac1{4\pi \text{Im}(z_0)}\right).
\end{equation}

By the uniqueness of $B(z_0)$, \eqref{eq:givesB} is the hypergeometric analog of \eqref{CS} and we can explicitly write the quasi-period $\int_C \nu$  as
$$ \frac 1{\pi}\int_C \nu \sim t(z_0)\frac{dF}{dt}\bigg|_{t=t(z_0)} \!\! +B(z_0)F(t(z_0)), $$  with $B(z_0)$ as above.

We further obtain the following lemma to certify that our hypergeometric evaluations are the periods and quasi-periods coming from the eigenvectors of the endomorphism ring of the associated CM elliptic curve $\widetilde{E}_d(X(z_0))$ given in  Corollary~\ref{cor:WIN5}.

\begin{lemma}\label{lem: eigenvector} 
Suppose $X(\tau)$ and $Z(\tau)= F_{0,(\frac1d,\frac{d-1}{d})}(X(\tau))^2$ satisfy all of the conditions in Corollary~\ref{cor:WIN5} with $\gamma=\left(\begin{smallmatrix}a&b\\c&-a\end{smallmatrix}\right)$, and $\tau_0=\frac{a}{c}+ \frac{i}{c\sqrt N}$ is a CM point with $\gamma\tau_0\in D$.
Then,
\[
2\alpha_N = \frac{1}{B(\gamma\tau_0)}  \in  \Q(X(\gamma\tau_0), X(\tau_0)).
\]
\end{lemma}

\begin{proof}
Recall we define $f':=\frac{1}{2\pi i}\frac{df}{d\tau}$.  Set $F(X):=F_{0,(\frac1d, \frac{d-1}d)}(X)$, and $Z(X)=F(X)^2$. Then $\frac{dF}{F}=\frac 12\frac{dZ}{Z}$. Since $\gamma\tau_0\in D$, we have from \eqref{eqn:constant c}, that 
\[
 -B(\gamma\tau_0)=\frac{X(\gamma\tau_0)}{X'(\gamma\tau_0)}\left( \frac12\frac{Z'(\gamma\tau_0)}{Z(\gamma\tau_0)}-\frac1{4\pi \text{Im}(\gamma\tau_0)}\right).
\]
Thus using \cite[(17)]{WIN5}, we obtain
\begin{align*}
-B(\gamma\tau_0)&=\frac{X(\gamma\tau_0)}{X'(\gamma\tau_0)}\left( \frac1{2}\cdot\frac1{2N}\left(\frac{c\sqrt N}{\pi}+\varepsilon\frac{1}{ N}M_N'(\tau_0)\right)-\frac{c}{4\pi \sqrt N}\right)\\
&=\frac{X(\gamma\tau_0)}{X'(\gamma\tau_0)}\frac1{4N}\frac{\varepsilon}{N}\frac{dM_N}{dX}\bigg\vert_{X=X(\tau_0)}X'(\tau_0)
   =X(\gamma\tau_0) \frac{X'(\tau_0)}{X'(\gamma\tau_0)}\frac{1}{4N}\frac{\varepsilon}{N} \frac{dM_N}{dX}\bigg\vert_{X=X(\tau_0)}\\
&=X(\gamma\tau_0)\frac{XUZ(\tau_0)}{XUZ(\gamma\tau_0)}\frac1{4N}\frac{\varepsilon}{N} \frac{dM_N}{dX}\bigg\vert_{X=X(\tau_0)}
=X(\gamma\tau_0)\frac{-N Z(\tau_0)XU(\tau_0)}{\varepsilon Z(\tau_0)XU(\gamma\tau_0)}\frac1{4N}\frac{\varepsilon}{N} \frac{dM_N}{dX}\bigg\vert_{X=X(\tau_0)}\\
    &=-\frac1{4N}\frac{U(\tau_0)}{U(\gamma\tau_0)} X(\tau_0)\frac{dM_N}{dX}\bigg\vert_{X=X(\tau_0)}
    =-\frac{v(\tau_0)}{4NU(\gamma\tau_0)}\\
    &= -\frac{1}{2\alpha_N}. 
\end{align*}
As in Remark \ref{rmk: forms}, we have in these cases that $U=1-X$, and thus the value $\alpha_N$ is indeed in the field $\Q(X(\gamma\tau_0), X(\tau_0))$. 
\end{proof}

We next observe the following useful lemma.

\begin{lemma} \label{lem:deltaquotioent}
Suppose $X(\tau)$ and $Z(\tau)= F_{0,(\frac12, \frac1d,\frac{d-1}{d})}(X(\tau))^2$ satisfy the conditions in Corollary \ref{cor:WIN5} with $\gamma=\left(\begin{smallmatrix}a&b\\c&-a\end{smallmatrix}\right)$, and $\tau_0=\frac{a}{c}+ \frac{i}{c\sqrt N}$ is a CM point with $\gamma\tau_0\in D$.  Then 
\[
\frac{\delta_N}{F_{0,(\frac12, \frac1d,\frac{d-1}{d})}(X(\gamma\tau_0))} \in \overline{\Q}.
\]
\end{lemma}

\begin{proof}
Recall that from the classical result of modular forms with CM-values  (see \cite[Proposition 5]{WIN2},  \cite[\S 5.10]{Cohen-Stromberg}, or \cite[Proposition 26]{Zagier} for example), if $f$ is a meromorphic modular form of weight $k$ with algebraic Fourier coefficients at a CM point $z_0\in K\cap \HH$, then 
$$
   f(z_0) \sim \Omega_K^k. 
$$
In the cases under consideration, the modular form $Z(\tau)$ is a weight-$4$ modular form whose Fourier expansion has integer coefficients. Therefore $Z(\gamma\tau_0) \sim \Omega_K^4$, where $K=\Q(c\sqrt{-N})$. Also, from the expression of $\delta_N$ in  Corollary  \ref{cor:WIN5} we have
$$
  U(\tau_0)X(\tau_0) \left(\frac{dM_N}{dX}\right) \bigg|_{X=X(\tau_0)} \!\!\! \cdot  \delta_N \sim 1,
$$
where 
$X(\tau_0)$, $ \left(\frac{dM_N}{dX}\right) \left|_{X=X(\tau_0)} \right.$ are algebraic numbers, and 
$UZ(\tau) = \frac{1}{X}\cdot \frac{1}{2\pi i}\frac{dX}{d\tau}$.  Since $UZ(\tau) = \frac{1}{X}\cdot \frac{1}{2\pi i}\frac{dX}{d\tau}$ is a meromorphic modular form of weight $2$ and has integer Fourier coefficients, we have
$$
   UZ(\tau_0)  \sim \Omega_K^2.  
$$
Together with the fact  
$Z(\tau_0)    \sim \Omega_K^4$, 
the identity tells us that 
$U(\tau_0)    \sim \Omega_K^{-2}$ and hence
$$
   \delta_N \sim   \Omega_K^2\sim F_{0,(\frac12, \frac1d,\frac{d-1}{d})}(X(\gamma\tau_0)). 
$$
\end{proof}

We now state the following crucial result which is an analogue of \cite[Prop. 16]{WIN2} and is used in the proof of Theorem \ref{thm:main}.  

\begin{proposition}\label{lem:main}
Suppose $X(\tau)$ and $Z(\tau)= F_{0,{\bf b}}(X(\tau))^2$ satisfy all of the conditions in Corollary \ref{cor:WIN5} with $\gamma=\left(\begin{smallmatrix}a&b\\c&-a\end{smallmatrix}\right)$, and $\tau_0=\frac{a}{c}+ \frac{i}{c\sqrt N}$ is a CM point with $\gamma\tau_0\in D$. Suppose further that $\tau_0, \lambda, \mu, p$ satisfy all of the conditions of Theorem \ref{thm:main}. Let $a(n)$ and $b(n)$ be defined by \eqref{def:wvInt}.  Then, we have that

\begin{align*}
a(p-1) \equiv  [F_{0,(\frac{1}{d}, \frac{d-1}{d})}]_{p-1}(\mu) \pmod{p \Z_p[\mu]},  \\
[F_{0,(\frac{1}{d}, \frac{d-1}{d})}]_{p-1}(\mu)\cdot[F_{c_\mu,(\frac{1}{d}, \frac{d-1}{d})}]_{p-1}(\mu) \equiv  a(p-1)b(p-1)  \pmod{p^2 \Z_p[\mu]}, 
\end{align*}

where
\[
c_\mu = \begin{cases}
         2\alpha_N & \text{ if }  \mu = \lambda \\ 
          \frac{4\alpha_N(1-\mu)}{1-2\mu}=   \frac{4\alpha_N(1-\mu)(1-2\mu)}{1-\l} & \text{ if } \mu = \frac{1 \pm \sqrt{1-\l}}{2}. 
    \end{cases}
\]
\end{proposition}

Before we prove this proposition, we need a couple facts about power series and derivatives. For a power series $f(z)=\sum_{k\geq 0} c_kz^k \in \C[[z]]$, 
\begin{equation}\label{eq:diff_gen}
\sum_{k\geq 0} \left(ak+1\right) c_k z^k = f(z)+ az\frac{d}{dz}f(z),
\end{equation}
and thus for the hypergeometric functions defined in \eqref{def:F_a},
\begin{equation}\label{eq:diff_basic}
F_{a,{\bf b}}(z) = F_{0,{\bf b}}(z)+ az\frac{d}{dz}F_{0,{\bf b}}(z).
\end{equation}

If $F_{0,{\bf b}}(z)^2 = \sum_{k\geq 0} A_k z^k$, then from \eqref{eq:diff_gen} and \eqref{eq:diff_basic} we have
\begin{equation}\label{eq:sqcoeffs}
\sum_{k\geq 0} \left(ak+1\right) A_k z^k = F_{0,{\bf b}}(z)^2 + 2azF_{0,{\bf b}}(z)\frac{d}{dz}F_{0,{\bf b}}(z) = F_{0,{\bf b}}(z) \cdot F_{2a,{\bf b}}(z).
\end{equation}

In \cite[(3.5)]{WIN2}, the authors show that for $\l$  such that both sides converge,
\begin{equation}\label{eq:geom_clausen}
F_{0,(\frac12,c,1-c)} (\lambda) = F_{0,(c,1-c)}\left(\frac{1\pm \sqrt{1-\l}}{2}\right)^2.
\end{equation}
This enables us to show the following.

\begin{lemma} \label{lem:4F3prod}
Set $t=\frac{1- \sqrt{1-\l}}{2}$, so that $\l = -4t(t-1)$. Then, when both sides converge,
\[
F_{a_1,{(\frac12,c,1-c))}}(\l) =  F_{0,(c,1-c)}(t) \cdot F_{a_2,{(c,1-c))}}(t),
\]
where $a_2 = \frac{2a_1(1-t)}{1-2t}$.
\end{lemma}

\begin{proof}
Using \eqref{eq:diff_basic} and \eqref{eq:geom_clausen}, we obtain 
\begin{align*}
F_{a_1,{(\frac12,c,1-c)}}(\l) &= F_{0,(c,1-c)}(t)^2 + a_1 \l\frac{d}{d\l}F_{0,(c,1-c)}(t)^2 \\
&= F_{0,(c,1-c)}(t)^2 + 2a_1 F_{0,(c,1-c)}(t)\cdot \l \frac{dt}{d\l} \cdot  \frac{d}{dt}F_{0,(c,1-c)}(t)\\
&= F_{0,(c,1-c)}(t)^2 + \frac{2a_1(1-t)}{1-2t} F_{0,(c,1-c)}(t)\cdot t \frac{d}{dt}F_{0,(c,1-c)}(t).
\end{align*}
Thus applying \eqref{eq:diff_gen} gives the desired result.
\end{proof}

We now prove Proposition \ref{lem:main}.

\begin{proof}[Proof of Prop. \ref{lem:main}]
As in \S \ref{CMcurves} and \eqref{def:wv}, consider the differentials $\omega_t=\sum_{n\ge 0} a_t(n)\xi^n d\xi$ and $\nu_t=\omega_t+ c_t \cdot  t \partial_t \omega_t=\sum_{n\ge 0} b_t(n)\xi^n d\xi$.  By \cite[Lemma 12]{WIN2}, 
\begin{equation}
\label{a(p-1)}
  a_t(p-1) \equiv [F_{0,(\frac{1}{d}, \frac{d-1}{d})}]_{p-1}(t) \pmod{p}.  
\end{equation}

On the other hand, 
\begin{align}
b_t(p-1) & = a_t(p-1)+c_t \cdot t \frac{d}{dt}a_t(p-1) \label{b(p-1)} \\
 & \equiv [F_{0,(\frac{1}{d}, \frac{d-1}{d})} + c_t \cdot t \frac{d}{dt}F_{0,(\frac{1}{d}, \frac{d-1}{d})}]_{p-1}(t) \pmod{p} \nonumber\\
 & \overset{\eqref{eq:diff_basic}}{\equiv} [F_{c_t ,(\frac{1}{d}, \frac{d-1}{d})}]_{p-1}(t) \pmod{p}. \nonumber
\end{align}
Next, we show the congruence relations at fixed $t=\mu$ where $\widetilde{E}_d(\mu)$ has CM. The differentials 
\[\omega:= \omega_\mu=\sum_{n\ge 0} a(n-1)\xi^n d\xi, \quad  \text{and} \quad \nu:= \nu_\mu=\sum_{n\ge 0} b(n-1)\xi^n d\xi\]
are non-parallel common eigenvectors for the action of $\End \,\, \widetilde{E}_d(\mu)$ on the first de Rham cohomology group $H_{\text{dR}}^1(\widetilde{E}_d(\mu)/A)$.  Here $A$ is the ring of Witt vectors $W(\mathbb{F}_p(\mu))$ as defined in \eqref{def:A}, which is isomorphic to the ring of integers $\Z_p[\mu]$ of $\Q_p(\mu)$. 

The first congruence relation then follows immediately from \eqref{a(p-1)}, and we obtain
\begin{equation}
\label{eq:specifica(p-1)}
  a(p-1) \equiv  [F_{0,(\frac{1}{d}, \frac{d-1}{d})}]_{p-1}(\mu) \pmod{pA},
\end{equation}
where $A=\Z_p[\mu]$ and we note that $A=\Z_p$ when $\mu=\l$. 

To show the second congruence relation, we examine the unique value $c_\mu  = \frac{1}{B(z_0)}$, where $\widetilde{E}_d(\mu)$ has CM by $z_0=\gamma\tau_0$ and $B(z_0)$ is as in \eqref{eq:givesB}. In order to write $c_\mu$ as a function of $\alpha_N$, note that since the value $\alpha_N$ depends on $Z(\tau)$ (see Corollary \ref{cor:WIN5}), and therefore on $\bf b$, then the expression of $c_\mu$ as a function of $\alpha_N$ will depend on whether $\mu=\lambda$ or $\mu=\frac{1 \pm \sqrt{1-\l}}{2}$.  When $\mu=\lambda$, the value $c_\mu=\frac{1}{B(z_0)}=2\alpha_N$ is given by Lemma~\ref{lem: eigenvector} and \eqref{eqn:constant c}. When $\mu=\frac{1 \pm \sqrt{1-\l}}{2}$, from \eqref{eq:givesB} we have that $c_\mu$ is the unique algebraic number such that 

\begin{align*}
   \frac{1}{\pi} & \quad \sim  \quad  F_{0,(\frac{1}{d}, \frac{d-1}{d})}(\mu) \cdot F_{c_\mu,(\frac{1}{d}, \frac{d-1}{d})}(\mu) \\
   & \overset{\text{Lem.} \ref{lem:4F3prod}}{=}  F_{c_\mu',(\frac{1}{2}, \frac{1}{d}, \frac{d-1}{d})}(\lambda),
\end{align*}
where $c_\mu'=\frac{c_\mu(1-2\mu)}{2(1-\mu)}$. From Corollary \ref{cor:WIN5}, we obtain 
\begin{equation} \label{eq:WIN2cond}
\frac{1}{\pi}\cdot \frac{\delta_N}{ F_{0,(\frac{1}{2}, \frac{1}{d}, \frac{d-1}{d})}(\lambda) } =  F_{2\alpha_N,(\frac{1}{2}, \frac{1}{d}, \frac{d-1}{d})}(\lambda).
\end{equation}
By Lemma \ref{lem:deltaquotioent}, $\frac{\delta_N}{F_{0,(\frac{1}{2}, \frac{1}{d}, \frac{d-1}{d})}(\lambda) }$ is algebraic, and the uniqueness of $c_\mu$ gives that $c_\mu'=2\alpha_N$, and thus $c_\mu=\frac{4\alpha_N(1-\mu)}{1-2\mu}$. Therefore, \ref{eq:specifica(p-1)} together with the specific expression of $c_\mu$  gives
$$[F_{0,(\frac{1}{d}, \frac{d-1}{d})}]_{p-1}(\mu)\cdot[F_{c_\mu,(\frac{1}{d}, \frac{d-1}{d})}]_{p-1}(\mu) \equiv  a(p-1)b(p-1)  \pmod{p A}.$$

In order to show the second congruence relation modulo $p^2$, we consider separately the cases when $\widetilde{E}_d(\mu)$ is supersingular or ordinary at $p$.  When $\widetilde{E}_d(\mu)$ is supersingular, the congruence follows as in \cite[Proof of Proposition~16]{WIN2} by using the action of the multiplication by $-p$  on $\widetilde{E}_d(\mu)$.  When $\widetilde{E}_d(\mu)$ is ordinary at $p$ we use the action of the multiplication by $u_p$ (the unit root given in Theorem \ref{thm:main}) on $\widetilde{E}_d(\mu)$ as in Remark \ref{rem: action}. Then the congruence follows from Case 1 of the proof of Theorem \ref{thm:WIN2b}  when  $\mu$ can be embedded in $\Z_p$, and from Case 2 of the proof of Theorem \ref{thm:WIN2b} when $\mu$ can not be embedded in $\Z_p$.  In each case, we obtain 

$$
  [F_{0,(\frac{1}{d}, \frac{d-1}{d})}]_{p-1}(\mu)\cdot [F_{c_\mu,(\frac{1}{d}, \frac{d-1}{d})}]_{p-1}(\mu)  \equiv  a(p-1)b(p-1) \equiv \sgn \cdot \varepsilon_\mu\, p \pmod{p^2A},
$$
where $\sgn= 1$ (or $-1$) if and only if  $\widetilde{E}_d(\mu)$ is ordinary (or supersingular) at $p$, and $\varepsilon_\mu=-1$ if $\mu$ can not be embedded in $\Z_p$ and $\varepsilon_\mu=1$ otherwise.

\end{proof}

\section{Proof of Theorem \ref{thm:main}}
\label{sec:proof}

Before we prove Theorem \ref{thm:main} we need a few results about $p$-divisibility of our hypergeometric coefficients and truncating hypergeometric products modulo $p^2$ and $p^3$.  We use the following notation.  For $d\in\{2,3,4,6\}$, and $p\geq 3$ prime, define positive integers $r_d$ and $s_d$ by
\begin{align}\label{def:rdsd}
r_d &:=
\begin{cases}
\frac{p-1}{d} &  \text{if } p\equiv 1 \pmod d \\
\frac{p-(d-1)}{d} &  \text{if } p\equiv -1 \pmod d
\end{cases} \\
s_d &:=
\begin{cases}
\frac{(d-1)p-(d-1)}{d} &  \text{if } p\equiv 1 \pmod d \\
\frac{(d-1)p-1}{d} &  \text{if } p\equiv -1 \pmod d.
\end{cases} \nonumber
\end{align}

We first note that for $p$ an odd prime and $k$ a nonnegative integer such that $\frac{p+1}{2}\leq k\leq p-1$,
\begin{equation} \label{eq:half}
\frac{\left( \frac{1}{2} \right)_k}{(1)_k} \equiv 0 \pmod{p}.
\end{equation}

Chisholm, et al. \cite[Lemma 17]{WIN2} further give that for $d \in \{2, 3, 4, 6\}$ and $p$ an odd prime coprime to $d$, 
\begin{equation} \label{lem:win2}
\frac{\left( \frac{1}{d} \right)_k \left( \frac{d-1}{d} \right)_k}{(1)_k^2} \equiv
\begin{cases}
0 \!\!\! \pmod{p}, \text{ but } \not\equiv 0 \!\!\! \pmod{p^2}  & \text{ if } r_d < k \leq s_d, \\
0 \!\!\! \pmod{p^2}  & \text{ if } s_d <k \leq p-1.
\end{cases}
\end{equation}
Considering $\sum_{k=0}^\infty c_kz^k = (\sum_{k=0}^\infty a_kz^k)(\sum_{k=0}^\infty b_kz^k)$, we observe that truncating at $p-1$ gives
\[
\sum_{k=0}^{p-1} c_kz^k = \sum_{k=0}^{p-1}\sum_{j=0}^k a_jb_{k-j}z^k,
\]
whereas truncating individually and taking the product gives
\begin{multline}\label{eq:truncprod}
\left(\sum_{k=0}^{p-1}  a_kz^k \right) \!\! \left(\sum_{k=0}^{p-1} b_kz^k\right) 
= \!\! \sum_{k=0}^{2(p-1)} \!\!\!\! \sum_{\substack{0\leq j \leq p-1 \\ 0\leq k-j \leq p-1}} \!\!\!\! a_jb_{k-j} z^k
= \sum_{k=0}^{p-1} \sum_{j=0}^k a_jb_{k-j} z^k + \!\! \sum_{k=p}^{2(p-1)} \sum_{j=k-(p-1)}^{p-1} \!\!\!\! a_jb_{k-j} z^k \\
= \sum_{k=0}^{p-1} c_kz^k + \sum_{k=1}^{p-1}  \sum_{j=k}^{{p-1}} a_jb_{(p-1)-(j-k)} z^{(p-1)+k} 
= \sum_{k=0}^{p-1} c_kz^k + \sum_{k=1}^{p-1}  \sum_{j=0}^{{p-1-k}} a_{k+j}b_{(p-1)-j} z^{(p-1)+k}.
\end{multline}
These observations allow us to prove the following lemma which is necessary for our proof of Theorem \ref{thm:main}.

\begin{lemma} \label{lem:errmodp^3}
Using the notation in Theorem \ref{thm:main}, we have that if $p\geq 5$ prime and $\alpha_N \in \overline{\Q}$ such that $\alpha_N$ can be embedded into $\Z_p$ (and fixing such an embedding),
\[
[F_{0,{\bf b}}\cdot F_{2\alpha_N, {\bf b}}]_{p-1}(z) \equiv [F_{0,{\bf b}}]_{p-1}(z) \cdot [F_{2\alpha_N, {\bf b}}]_{p-1}(z) \;  \begin{cases} \!\!\!\! \pmod{p^2} \text{ if } {\bf b} = (\frac{1}{d}, \frac{d-1}{d}), \\ \!\!\!\! \pmod{p^3} \text{ if } {\bf b} = (\frac12, \frac{1}{d}, \frac{d-1}{d}). \end{cases}
\]       
\end{lemma}
Note that letting $\mu=\frac{1 \pm \sqrt{1-\l}}{2}$, it follows from Lemmas \ref{lem:4F3prod}, \ref{lem:errmodp^3}, and Proposition \ref{lem:main} that
\[
[F_{2\alpha_N,(\frac12,\frac{1}{d}, \frac{d-1}{d})}]_{p-1}(\l) \equiv [F_{0,(\frac{1}{d}, \frac{d-1}{d})}]_{p-1}(\mu) \cdot [F_{c_\mu,(\frac{1}{d}, \frac{d-1}{d})}]_{p-1}(\mu)  \equiv  \sgn \cdot \varepsilon_\mu \, p \pmod{p^2 A}.
\]
Thus, since $[F_{2\alpha_N,(\frac12,\frac{1}{d}, \frac{d-1}{d})}]_{p-1}(\l)$ and $\sgn \cdot \varepsilon_\mu \, p$ are both in $\Z_p$, we have from Proposition \ref{lem:main} that 
\begin{align*}
[F_{2\alpha_N,(\frac12,\frac{1}{d}, \frac{d-1}{d})}]_{p-1}(\l) \equiv \sgn \cdot \varepsilon_\mu \, p \pmod{p^2}.
\end{align*}

\begin{proof}[Proof of Lemma \ref{lem:errmodp^3}]
Let $d\in \{2,3,4,6\}$ and set ${\bf b} = (\frac{1}{d}, \frac{d-1}{d})$, or ${\bf b} = (\frac12, \frac{1}{d}, \frac{d-1}{d})$. Write $F_{0,{\bf b}}(z) = \sum_{k\geq 0} a_kz^k$ and $F_{2\alpha_N,{\bf b}}(z) = \sum_{k\geq 0} b_kz^k$.  Then by \eqref{eq:truncprod}, we have
\[
[F_{0,{\bf b}}\cdot F_{2\alpha_N, {\bf b}}]_{p-1}(z) = [F_{0,{\bf b}}]_{p-1}(z)\cdot[F_{2\alpha_N, {\bf b}}]_{p-1}(z) - \sum_{k=1}^{p-1}  \sum_{j=0}^{{p-1-k}} a_{k+j}b_{(p-1)-j} z^{(p-1)+k}.
\]
It follows from  \eqref{def:F_a} that $b_k=(2\alpha_N k+1)a_k$. Since $\alpha_N$ can be embedded into $\Z_p$ it suffices to show for $1\leq k \leq p-1$ and $0\leq j\leq p-1-k$ that $a_{k+j}a_{(p-1)-j}$ is congruent to $0$ modulo our desired power of $p$.
 
Using the notation in \eqref{lem:win2}, we consider when $p-1-j$ is in the intervals $(0,r_d]$, $(r_d,s_d]$, and $(s_d,p-1]$ separately. 

First, we consider the case when ${\bf b} = (\frac{1}{d}, \frac{d-1}{d})$, so that $a_k=\frac{\left( \frac{1}{d}\right)_k\left( \frac{d-1}{d}\right)_k}{(1)_k^2}$.  When $p-1-j\in (s_d,p-1]$, \eqref{lem:win2} directly gives that $a_{p-1-j}\equiv 0 \pmod{p^2}$.  When $p-1-j\in (0,r_d]$, then $s_d=p-1-r_d \leq j < p-1$, which implies that $s_d<j+k \leq p-1$, so \eqref{lem:win2} gives that $a_{k+j}\equiv 0 \pmod{p^2}$. Now suppose $p-1-j\in (r_d,s_d]$.  Then, we have 
\[
r_d=p-1-s_d\leq j < p-1-r_d=s_d.
\]
Thus, we have that $p-1-j\in (r_d,s_d]$ and also that $j+k\in (r_d,p-1]$, so \eqref{lem:win2} gives that $a_{p-1-j}\equiv 0 \pmod{p}$ and $a_{k+j}\equiv 0 \pmod{p}$.

Next, we consider the case when ${\bf b} = (\frac12, \frac{1}{d}, \frac{d-1}{d})$, so that $a_k=\frac{\left( \frac{1}{2}\right)_k\left( \frac{1}{d}\right)_k\left( \frac{d-1}{d}\right)_k}{(1)_k^3}$.  From the previous case, we have seen that $\frac{\left( \frac{1}{d}\right)_k\left( \frac{d-1}{d}\right)_k}{(1)_k^2}\equiv 0\pmod{p^2}$ for $p-1-j\in (0,p-1]$.  Thus it suffices to show that when $p-1-j\in (0,p-1]$, either $\frac{\left( \frac{1}{2} \right)_{p-1-j}}{(1)_{p-1-j}} \equiv 0 \pmod{p}$ or $\frac{\left( \frac{1}{2} \right)_{k+j}}{(1)_{k+j}} \equiv 0 \pmod{p}$.  If $p-1-j\in (\frac{p-1}{2},p-1]$, then from \eqref{eq:half} we have $\frac{\left( \frac{1}{2} \right)_{p-1-j}}{(1)_{p-1-j}} \equiv 0 \pmod{p}$.  If $p-1-j\in (0,\frac{p-1}{2}]$, then $\frac{p-1}{2} \leq j <p-1$, so we have that $\frac{p-1}{2} < k+j \leq p-1$, and so from \eqref{eq:half} we have $\frac{\left( \frac{1}{2} \right)_{k+j}}{(1)_{k+j}} \equiv 0 \pmod{p}$.
\end{proof}

We now prove our main theorem.

\begin{proof}[Proof of Theorem \ref{thm:main}]

Let $d\in \{2,3,4,6\}$ and set ${\bf b} = (\frac{1}{d}, \frac{d-1}{d})$, or ${\bf b} = (\frac12, \frac{1}{d}, \frac{d-1}{d})$. Suppose $Z(\tau)= F_{0,{\bf b}}(X(\tau))^2$ satisfies all of the conditions in Corollary \ref{cor:WIN5}, and that if $\l = X(\gamma\tau_0)$, then $F_{0,{\bf b}}(\l) = F_{0,(\frac{1}{d}, \frac{d-1}{d})}(\mu)^r$ for some $r$, where $\mu\in \{\l, \frac{1 \pm \sqrt{1-\l}}{2}\}$.

Assume that $\Q(\mu) \subseteq \Q(\sqrt{1-\l})$ is a totally real field, $|\l|, |\mu|<1$,  and that $\widetilde E_d(\mu)$ has CM.  Let $p$ be an odd prime such that $X(\tau_0), \lambda, \mu$ can be embedded into $\Z_p$, and $p$ is unramified in $\Q(\sqrt{1-\l})$.  Suppose that $\widetilde E_d(\mu)$ has good reduction modulo $p$. 

First, we consider the case when ${\bf b} = (\frac{1}{d}, \frac{d-1}{d})$ and $\mu=\l=X(\gamma\tau_0)$, so that $F_{0,{\bf b}}(X(\gamma\tau_0)) = F_{0,(\frac{1}{d}, \frac{d-1}{d})}(\mu)$.  Then by Lemma \ref{lem:errmodp^3} and Proposition \ref{lem:main},

\begin{align*}
[F_{0,{\bf b}}\cdot F_{2\alpha_N,{\bf b}}]_{p-1}(X(\gamma\tau_0)) & = [F_{0,(\frac{1}{d}, \frac{d-1}{d})}\cdot F_{2\alpha_N,(\frac{1}{d}, \frac{d-1}{d})}]_{p-1}(\mu) \\
& \equiv [F_{0,(\frac{1}{d}, \frac{d-1}{d})}]_{p-1}(\mu) \cdot [F_{2\alpha_N,(\frac{1}{d}, \frac{d-1}{d})}]_{p-1}(\mu) \pmod{p^2} \\ 
& \equiv a(p-1)b(p-1)\pmod{p^2}.
\end{align*}

Applying Lemma \ref{lem:WIN2Prop16}, we thus obtain
\[
[F_{0,{\bf b}}\cdot F_{2\alpha_N,{\bf b}}]_{p-1}(X(\gamma\tau_0)) \equiv \begin{cases} p &\text{ if } \widetilde{E}_d(\lambda) \text{ is ordinary at }p \\ -p &\text{ if } \widetilde{E}_d(\lambda) \text{ is supersingular at }p \end{cases} \pmod{p^2}. 
\]

Applying a theorem of Deuring \cite[Ch. 13 Thm. 12]{Lang_EF}, $\widetilde E_d(\l)$ is supersingular at $p$ if and only if $p$ either ramifies or is inert in $k=\Q(\tau_0)=\Q(c\sqrt{-N})$, i.e., $E$ is ordinary modulo $p$ if and only if $p$ splits in $k$.  Thus by the splitting of primes in imaginary quadratic fields (see \cite[Thm. 25]{Marcus}), we obtain that $\widetilde E_d(\l)$ is ordinary if and only if $\left( \frac{-c^2N}{p} \right ) =1$, as desired. 

Next, we consider the case when ${\bf b} = (\frac12, \frac{1}{d}, \frac{d-1}{d})$ and $\mu = \frac{1 \pm \sqrt{1-\l}}{2}=\frac{1 \pm \sqrt{1-X(\gamma\tau_0)}}{2}$, so that $F_{0,{\bf b}}(X(\gamma\tau_0)) = F_{0,(\frac{1}{d}, \frac{d-1}{d})}(\frac{1 \pm \sqrt{1-X(\gamma\tau_0)}}{2})^2$.  Then by Lemma \ref{lem:errmodp^3},

\begin{align*}
[F_{0,{\bf b}}\cdot F_{2\alpha_N,{\bf b}}]_{p-1}(X(\gamma\tau_0)) & = [F_{0,(\frac12, \frac{1}{d}, \frac{d-1}{d})}\cdot F_{2\alpha_N,(\frac12,\frac{1}{d}, \frac{d-1}{d})}]_{p-1}(\l) \\
& \equiv [F_{0,(\frac12,\frac{1}{d}, \frac{d-1}{d})}]_{p-1}(\l)\cdot [F_{2\alpha_N,(\frac12,\frac{1}{d}, \frac{d-1}{d})}]_{p-1}(\l) \pmod{p^3}.
\end{align*}

From \eqref{eq:WIN2cond} we know that $\alpha, \lambda$ admit a series of the form \eqref{eq:ram} which satisfies the conditions of Theorems \ref{thm:WIN2a} and \ref{thm:WIN2b}.  Thus we have that
\[
[F_{2\alpha_N,(\frac12,\frac{1}{d}, \frac{d-1}{d})}]_{p-1}(\l) \equiv 
\begin{cases}
\left( \frac{1-\l}p \right) \cdot p  & \text{ if } \widetilde{E}_d(\mu) \text{ is ordinary at }p \\
- \left( \frac{1-\l}p \right) \cdot p  & \text{ if } \widetilde{E}_d(\mu) \text{ is supersingular at }p
\end{cases}
\pmod {p^2},
\]
and that
\[
[F_{0,(\frac12,\frac{1}{d}, \frac{d-1}{d})}]_{p-1}(\l) \equiv  
\begin{cases}
u_p^2  & \text{if } \widetilde{E}_d(\frac{1-\sqrt{1-\lambda}}{2}) \text{ is ordinary at } p \text{ and }\left(\frac{1-\l}p \right)=1 \\
\left( \frac{k_d}{p} \right)u_p^2  & \text{if } \widetilde{E}_d(\frac{1-\sqrt{1-\lambda}}{2}) \text{ is ordinary at } p \text{ and }\left(\frac{1-\l}p \right)=-1 \\
0 & \text{if } \widetilde{E}_d(\mu) \text{ is supersingular at } p
\end{cases}
\pmod {p^2}.
\]
Putting these together and again applying \cite[Ch. 13 Thm. 12]{Lang_EF} gives the desired result.
\end{proof}

\section{Examples}
\label{sec:examples}
\renewcommand{\theequation}{\Alph{equation}}
Below we give 11 examples of Theorem \ref{thm:main} and 4 examples of Theorem \ref{thm:WIN2a}, each of which stems from series evaluations in \cite[Thm. 2]{WIN5}. While each example is unique, we note that those in \S \ref{sec:win2examples} arise from the same series as some of those in \S \ref{examples4.2-5}.

As each of the following examples stem from \cite[Thm. 2]{WIN5}, it is clear that Corollary \ref{cor:WIN5} will hold.  However, in an effort to keep this paper self-contained we will provide the data required to demonstrate that the conditions of Corollary \ref{cor:WIN5} are satisfied for each example, giving that 
\[ \frac{\delta}{\pi} = F_{0,{\bf b}} (\lambda) \cdot F_{2\alpha,{\bf b}} (\lambda),\]
where $\delta, \alpha, \lambda$ are given in Table~\ref{tab:RSData}. This table also lists the values of $N, \lambda, \mu, d$ necessary for Theorem \ref{thm:main} from which we can see that for each example $\Q (\lambda)$ is totally real and $|\lambda|,|\mu| <1$. 
We then determine primes $p$ for which Theorem \ref{thm:main} holds. 
In particular, we ensure that $X(\tau_0)$ and $\lambda$ can be embedded into $\mathbb{Z}_p$, that $p$ is unramified in $\Q(\mu)$ and determine the primes of good reduction for the corresponding elliptic curve $\widetilde{E}_d(\mu)$.  Data for the elliptic curves $\widetilde{E}_d(\mu)$ is obtained using \texttt{Sagemath} \cite{sagemath}. We note that each elliptic curve $\widetilde{E}_d(\mu)$ has CM. 

The Hauptmoduln are written as in \cite{WIN5} in terms of the $j$-function and Dedekind's $\eta$-function, 
\begin{align*}
    t_2(\tau) &= -64\eta(2\tau)^{24}/\eta(\tau)^{24}\\
    t_3(\tau) &= -27\eta(3\tau)^{12}/\eta(\tau)^{12}\\
    t_\infty(\tau) &= 16\eta(\tau)^8\eta(4\tau)^{16}/\eta(2\tau)^{24}\\
    t_{2,3}(\tau) &=1728/j(\tau)\\
    t_{2,4}(\tau) &=256\eta(\tau)^{24}\eta(2\tau)^{24}/(\eta(\tau)^{24}+64\eta(2\tau)^{24})^2\\
    t_{2,6}(\tau) &= 108\eta(\tau)^{12}\eta(3\tau)^{12}/(\eta(\tau)^{12}+27\eta(3\tau)^{12})^2.
\end{align*}
For a discussion regarding our choice of Hauptmoduln, see \cite{WIN5}.

\subsection{Examples from \texorpdfstring{\cite[\S 4]{WIN5}}{}} \label{examples4.2-5} 

In \eqref{eq:example4.2a}-\eqref{eq:example4.7}, we list examples of Theorem \ref{thm:main} stemming from from \cite[\S4]{WIN5}. For each, we have that $Z(\tau) = F_{0,(\frac{1}{d},\frac{d-1}{d})}(X(\tau))^2$ where $d$ is listed in Table \ref{tab:RSData} and $X$ is given in Table \ref{tab:Sec4ExampleData}, so we have that $\mu = \lambda$. 

In Table \ref{tab:Sec4ExampleData}, we list the  label from \cite{WIN5} associated to each example, the data required to construct the Ramanujan-Sato series in Corollary \ref{cor:WIN5}, and $d$ such that $\mathbf{b} = \left(\frac 1d, \frac{d-1}{d}\right)$. Note that for each of these examples, $U(\tau) = 1-X(\tau)$, the accompanying $N$ can be found in Table \ref{tab:RSData}, and one can determine the value of $\tau_0$ using Corollary \ref{cor:WIN5}.

\begingroup
\renewcommand{\arraystretch}{1.6}
\begin{table}[h] \label{tab:Sec4ExampleData}
    \centering
    \begin{tabular}{|c|c|c|c|c|c|}
    \hline
        Example & Reference from \cite{WIN5} &  $\gamma$ & $\Gamma$&   $X(\tau)$ & $\varepsilon$\\\hline
        
       \eqref{eq:example4.2a} & 4.2  & \multirow{2}{*}{$ \left( \begin{smallmatrix}  0 & -\frac{1}{\sqrt{2}} \\ \sqrt{2} & 0\end{smallmatrix}\right)$} &\multirow{4}{*}{$\Gamma_0(2)$}  &\multirow{4}{*}{$\frac{t_2(\tau)}{t_2(\tau)-1}$} & \multirow{2}{*}{$-1$} \\\cline{0-1}
       
        \eqref{eq:example4.3a}&4.3&  & &  & \\\cline{0-2}\cline{6-6}
        
        \eqref{eq:example4.4a}&4.4&  \multirow{2}{*}{$\left( \begin{smallmatrix} 1 & -1 \\ 2 & -1 \end{smallmatrix}\right)$} & & &   \multirow{2}{*}{$1$}\\\cline{0-1}
        
        \eqref{eq:example4.5a} &4.5 &  & &  & \\\cline{0-5}
        
        \eqref{eq:example4.6} &4.6 &  $\left( \begin{smallmatrix}
            0 & -\frac{1}{\sqrt{3}}\\ \sqrt{3} & 0 
        \end{smallmatrix}\right)$ & $\Gamma_0(3)$ & $\frac{t_3(\tau)}{t_3(\tau)-1}$ &  $1$\\\cline{0-5}
        
        \eqref{eq:example4.7} &4.7 & $\left( \begin{smallmatrix}
            0 & -\frac{1}{2}\\ 2 & 0 
        \end{smallmatrix}\right)$& $\Gamma_0(4)$ & $t_\infty(\tau)$ & $-1$\\\cline{0-5}
    \end{tabular}    
    \medskip
    \caption{Data required for examples \eqref{eq:example4.2a}-\eqref{eq:example4.7}}
\end{table}
\endgroup
Note that for each of these examples, we have that either $\Q(\lambda) = \Q$ or $\Q(\lambda) = \Q(\sqrt{D})$ for some non-square $D\in \Z$. In the case where $\Q(\lambda)$ is quadratic, requiring that $\lambda$ can be embedded into $\Z_p$ gives that $\left(\frac{D}{p}\right) = 1$, which in turn implies that $p$ splits in $\Q(\lambda)$. Thus, requiring that $\lambda$ can be embedded into $\Z_p$ is a stronger condition than requiring that $p$ is unramified in $\Q(\lambda)$. 
\noindent We obtain the following examples.  For $p\equiv \pm 1 \pmod 8$,
\begin{equation}\label{eq:example4.2a}
\left[F_{0,(\frac{1}{4}, \frac{3}{4})}\cdot F_{6\sqrt{2}+8,(\frac{1}{4}, \frac{3}{4})}\right]_{p-1}\left(\frac{3-2\sqrt{2}}{6}\right) \equiv 
\begin{cases}
p & \text{ if } \left ( \frac{-6}{p} \right ) =1 \\
-p & \text{ otherwise }
\end{cases}
\pmod{p^2}.
\end{equation}
For $p\equiv \pm 1 \pmod 5$, 
\begin{equation}\label{eq:example4.3a}
    \left[F_{0,(\frac{1}{4}, \frac{3}{4})}\cdot F_{\frac{20+9\sqrt{5}}{2},(\frac{1}{4}, \frac{3}{4})}\right]_{p-1}\left(\frac{9-4\sqrt{5}}{18}\right) \equiv 
\begin{cases}
p & \text{ if } \left ( \frac{-10}{p} \right ) =1 \\
-p & \text{ otherwise }
\end{cases}
\pmod{p^2}.
\end{equation}
For $p\geq 5,$
\begin{equation}
    \label{eq:example4.4a}
    \left[F_{0,(\frac{1}{4}, \frac{3}{4})}\cdot F_{8,(\frac{1}{4}, \frac{3}{4})}\right]_{p-1}\left(\frac{-1}{3}\right) \equiv 
\begin{cases}
p & \text{ if } \left ( \frac{-12}{p} \right ) =1 \\
-p & \text{ otherwise }
\end{cases}
\pmod{p^2}.
\end{equation}
For $p\equiv \pm 1 \pmod 5$,
\begin{equation}
    \label{eq:example4.5a}
    \left[F_{0,(\frac{1}{4}, \frac{3}{4})}\cdot F_{\frac{2}{3}(10+4\sqrt{5}),(\frac{1}{4}, \frac{3}{4})}\right]_{p-1}\left(\frac{2-\sqrt{5}}{4}\right) \equiv 
\begin{cases}
p & \text{ if } \left ( \frac{-20}{p} \right ) =1 \\
-p & \text{ otherwise }
\end{cases}
\pmod{p^2}.
\end{equation}
For $p\equiv \pm 1 \pmod 8$,
\begin{equation}
    \label{eq:example4.6}
    \left[F_{0,(\frac{1}{3}, \frac{2}{3})}\cdot F_{6(\sqrt{2}+1),(\frac{1}{3}, \frac{2}{3})}\right]_{p-1}\left(\frac{2-\sqrt{2}}{4}\right) \equiv 
\begin{cases}
p & \text{ if } \left ( \frac{-6}{p} \right ) =1 \\
-p & \text{ otherwise }
\end{cases}
\pmod{p^2}.
\end{equation}
For $p\equiv \pm 1 \pmod8$,
\begin{equation}
    \label{eq:example4.7}
    \left[F_{0,(\frac{1}{2}, \frac{1}{2})}\cdot F_{2(4+2\sqrt{2}),(\frac{1}{2}, \frac{1}{2})}\right]_{p-1}\left(3-2\sqrt{2}\right) \equiv 
\begin{cases}
p & \text{ if } \left ( \frac{-8}{p} \right ) =1 \\
-p & \text{ otherwise }
\end{cases}
\pmod{p^2}.
\end{equation}

\begin{remark}\label{rmk:win2examples}
Examples \eqref{eq:example4.2a}-\eqref{eq:example4.5a} arise from series that are also of the form $\frac{\delta}{\pi} = F_{\alpha,\bf{b}}(\lambda)$ and thus one could also obtain a supercongruence by applying Theorem \ref{thm:WIN2a} directly.  We do this in \S \ref{sec:win2examples} which yields distinct examples. 
\end{remark}

\subsection{Examples from \texorpdfstring{\cite[\S 5]{WIN5}}{} }\label{examples5}
\color{black}
In \eqref{eq:example5.2}-\eqref{eq:example5.6} we list examples of Theorem \ref{thm:main} stemming from \cite[Ex. 5.2-5.6]{WIN5}. 
For each example, we have  $Z(\tau) = F_{0,\left(\frac 12,\frac 12-\frac 1m, \frac 12+\frac 1m  \right)}(X(\tau))^2$ where $X(\tau) = t_{2,m}$ for $m\in \{3,4,6\}$.
Therefore we have ${\bf b} = (\frac{1}{2}, \frac{1}{d},\frac{d-1}{d})$ for $d=\frac{2m}{m-2} \in \{3,4,6\}$, and so $\mu = \frac12(1- \sqrt{1-\lambda})$ as in Table~\ref{tab:RSData}. Additionally, we have $\lambda = X(\gamma\tau_0) = X(\tau_0)$, which is given in Table \ref{tab:RSData}, as is the value for $N$. The value for each $\tau_0$ can be determined from Corollary \ref{cor:WIN5}.
In Table \ref{tab:Sec5ExampleData} we provide the label for the associated example from \cite{WIN5} as well as the remaining data required to construct the Ramanujan-Sato series in Corollary \ref{cor:WIN5}.     Further, we have that $\frac{1}{2\pi i} \frac{dX}{d\tau} = Z(\tau)X(\tau)U(\tau)$ where $U(\tau)$ is given in \cite[\S 5, (42)-(44)]{WIN5} and we consider  $\gamma=\frac{1}{\sqrt{s}}\left(\begin{smallmatrix}
0 & -1\\ s & 0 \\
\end{smallmatrix}\right)$. In each case, $(Z|_{\gamma})(\tau)=Z(\tau)$. We note that by applying \cite[(34)]{WIN2}) and \cite[(31)]{WIN2},
\[
Z(\tau) = F_{0,\left(\frac 12,\frac 12-\frac 1m, \frac 12+\frac 1m  \right)}(X(\tau))^2 =F_{0,\left(\frac 12-\frac 1m, \frac 12+\frac 1m  \right)}\left(\dfrac{1\pm \sqrt{1-X(\tau)}}{2}\right)^4 .
\]

For each example, we ensure that $X(\tau_0)$ (and $\lambda$) can be embedded into $\Z_p$, that $p$ is unramified in $\Q(\mu)$, and that $\widetilde{E}_d(\mu)$ has good reduction at $p$. Below $u_p$ is the unit root of the geometric Frobenius at $p$ acting on the first cohomology of the elliptic curve $\widetilde E_d(\mu)$ as defined in Theorem~\ref{thm:main}.
We note that in \eqref{eq:example5.5} and  \eqref{eq:example5.6} we were able to collapse the conditions on the congruence by using properties of the Legendre relation. 

\begingroup
\renewcommand{\arraystretch}{1.6}
\begin{table}[h]
    \centering
    \begin{tabular}{|c|c|c|c|c|}
    \hline
       Example  & Reference from \cite{WIN5} & $\Gamma$ & $m$ & $s$ \\\hline
       \eqref{eq:example5.2} & 5.2 & \multirow{2}{*}{${\rm PSL}_2(\Z)$} & \multirow{2}{*}{$3$} &  \multirow{2}{*}{$1$} \\\cline{0-1}
        \eqref{eq:example5.3} & 5.3 & & & \\\cline{0-4}
        \eqref{eq:example5.4} & 5.4 & $\Gamma_0^+(2)$ & $4$ & $2$\\\cline{0-4}
        \eqref{eq:example5.5} & 5.5 & \multirow{2}{*}{$\Gamma_0^+(3)$} & \multirow{2}{*}{$6$} & \multirow{2}{*}{$3$}\\\cline{0-1}
        \eqref{eq:example5.6} & 5.6 & & & \\\cline{0-4}
    \end{tabular}
    \medskip
    \caption{Data required for examples \eqref{eq:example5.2}-\eqref{eq:example5.6}}
    \label{tab:Sec5ExampleData}
\end{table}
\endgroup

\noindent We obtain the following examples. For $p\neq 2,5,$
\begin{multline}
    \label{eq:example5.2}
    \left[F_{0,(\frac12, \frac{1}{6}, \frac{5}{6})}\cdot F_{\frac{28}{3},(\frac12, \frac{1}{6}, \frac{5}{6})}\right]_{p-1}\!\left(\frac{27}{125}\right) \\ \equiv 
\begin{dcases}
pu_p^2 & \text{ if } \left( \frac{-2}{p} \right) =1,  \left( \frac{1-\lambda}{p} \right) =1\\
-\left(\frac{-1}{p}\right)pu_p^2 & \text{ if } \left( \frac{-2}{p} \right) =1, \left( \frac{1-\lambda}{p} \right) =-1\\
0 & \text{ otherwise }
\end{dcases}
\pmod{p^2}.
\end{multline}
For $p\geq 7$, 
\begin{equation}
    \label{eq:example5.3}
    \left[F_{0,(\frac12, \frac{1}{6}, \frac{5}{6})}\cdot F_{11,(\frac12, \frac{1}{6}, \frac{5}{6})}\right]_{p-1}\left(\frac{4}{125}\right) \equiv 
\begin{dcases}
pu_p^2 & \text{ if } \left( \frac{-3}{p} \right) =1, \left( \frac{5}{p} \right)=1\\
-\left(\frac{-1}{p}\right)pu_p^2 & \text{ if } \left( \frac{-3}{p} \right) =1, \left( \frac{5}{p} \right)=-1\\
0 & \text{ otherwise }
\end{dcases}
\pmod{p^2}.
\end{equation}
For $p\geq 5$, 
\begin{equation}
    \label{eq:example5.4}
    \left[F_{0,(\frac12, \frac{1}{4}, \frac{3}{4})}\cdot F_{8,(\frac12, \frac{1}{4}, \frac{3}{4})}\right]_{p-1}\left(\frac{1}{9}\right) \equiv 
\begin{dcases}
pu_p^2 & \text{ if } \left ( \frac{-6}{p} \right ) =1, \left(\frac{2}{p}\right) = 1\\
-\left(\frac{-2}{p}\right)pu_p^2 & \text{ if } \left ( \frac{-6}{p} \right ) =1, \left(\frac{2}{p}\right) = -1\\
0 & \text{ otherwise }
\end{dcases}
\pmod{p^2}.
\end{equation}
For $p\geq 5$, 
\begin{equation}
    \label{eq:example5.5}
    \left[F_{0,(\frac12, \frac{1}{3}, \frac{2}{3})}\cdot F_{6,(\frac12, \frac{1}{3}, \frac{2}{3})}\right]_{p-1}\left(\frac{1}{2}\right) \equiv 
\begin{dcases}
pu_p^2 & \text{ if } \left ( \frac{-6}{p} \right ) =1 \\
0 & \text{ otherwise }
\end{dcases}
\pmod{p^2}.
\end{equation}
For $p\neq 3,5$, 
\begin{equation}
    \label{eq:example5.6}
     \left[F_{0,(\frac12, \frac{1}{3}, \frac{2}{3})}\cdot F_{\frac{33}{4},(\frac12, \frac{1}{3}, \frac{2}{3})}\right]_{p-1}\left(\frac{4}{125}\right) \equiv 
\begin{dcases}
pu_p^2 & \text{ if } \left ( \frac{-15}{p} \right )=1\\
0 & \text{ otherwise }
\end{dcases}
\pmod{p^2}.
\end{equation}
Remark \ref{rem: action} notes that certain supercongruences arising from Theorem \ref{thm:main} hold modulo $p^3$. Examples \ref{eq:example5.2} and \ref{eq:example5.3} are of this type. See Section \ref{sec: modp3} for details and discussion.

\subsection{Examples of Theorem \ref{thm:WIN2a}}\label{sec:win2examples}
As described in Remark \ref{rmk: forms}, the examples in \cite[Ex. 4.2-4.5]{WIN5}, which are obtained from applying \cite[Cor. 4.1]{WIN5}, differ in form from their counterparts in Table \ref{tab:RSData}, which are obtained by applying Corollary \ref{cor:WIN5}, although only by an algebraic constant scaling factor.  However, as mentioned in Remark \ref{rmk:win2examples} we can directly apply Theorem \ref{thm:WIN2a} to the series in \cite[Ex. 4.2-4.5]{WIN5} to obtain supercongruences that are genuinely different from those obtained from Theorem \ref{thm:main}.  We note that this produces (finitely) more restrictions on the allowable primes, however the result is also expected to be stronger as Theorem \ref{thm:WIN2a} is conjectured to hold modulo $p^3$.

To explain this difference in outcomes, when applying Corollary \ref{cor:WIN5} we obtain a series of the form $F_{c,\hat{{\bf b}}}(\hat{\lambda})$ that is also a product $F_{0,{\bf b}}(\lambda) \cdot F_{a,{\bf  b}}(\lambda)$ with a quadratic relationship between $\hat{\lambda}$ and $\lambda$ (obtained through quadratic hypergeometric formulas) as in Lemma \ref{lem:4F3prod}. 
There is an algebraic number $\mathfrak{c}$ such that 
\[F_{c,\hat{{\bf b}}}(\hat{\lambda})= \mathfrak{c}F_{0,{\bf b}}\cdot F_{a,{\bf b}}(\lambda).\] 
However, as the series have different coefficients the evaluations of their truncations are not equal,
\[
[F_{c,\hat{{\bf b}}}]_{p-1}(\hat{\lambda})\neq [F_{0,{\bf b}}\cdot F_{a,{\bf b}}]_{p-1}(\lambda),
\]
and due to the quadratic relationship between $\lambda$ and $\hat{\lambda}$ these evaluations are not always congruent modulo $p^2$, as the truncated series $[F_{c,{\bf b}}]_{p-1}(\hat{\lambda})$ yields a higher exponent of $\lambda$ for the $n^\text{th}$ term.  Moreover, the truncations $[F_{c,{\bf b}}]_{p-1}(\hat{\lambda})$ and $[F_{0,{\bf b}}\cdot F_{a,{\bf b}}]_{p-1}(\lambda)$ are related to periods of isogenous, yet distinct, elliptic curves.

In the following examples, recall that $\sgn=\pm 1$, equaling $1$ (or $-1$) if and only if  $E_d(\lambda) = \widetilde{E}_d(\frac{1\pm\sqrt{1-\lambda}}{2})$ is ordinary (or supersingular) at $p$. Since $d = 2$ for each case, we write $\mathbf{b} = (\frac 12, \frac 12, \frac 12)$. For each example one can verify that the conditions of Theorem \ref{thm:WIN2a} hold. Namely, the values of $a$ and $\lambda$ are obtained directly from the series, and we ensure that $\Q(\lambda)$ is totally real, $|\lambda|<1$, and $E_2(\lambda)$ has CM. Further, we restrict to primes $p$ that are coprime to the discriminant of $\widetilde E_d(\lambda)$, unramified in $\Q(\sqrt{1-\lambda})$, and such that $a, \lambda$ can be embedded in $\Z_p^\times$.  As in \S \ref{examples4.2-5} and \S \ref{examples5}, data for each elliptic curve $\widetilde E_2(\lambda)$ is found using \texttt{Sagemath} \cite{sagemath}.

The series in \cite[Exs. 4.2-4.5]{WIN5} can be written as
\begin{align*}
\frac{2}{\pi(4\sqrt{2}-5)} & = F_{\frac{12}{3-2\sqrt{2}},\mathbf{b}}(12\sqrt{2}-17), \\
\frac{2}{\pi(23-10\sqrt{5})} & =F_{\frac{12(15+4\sqrt{5})}{29},{ \bf b}}(72\sqrt{5}-161), \\
\frac{4}{\pi} & = F_{6,{\bf b}}\left(\frac 14\right), \\
\frac{10(\sqrt{5}+1)}{\pi\sqrt{\sqrt{5}-2}} & = F_{5+\sqrt{5},{\bf b}}(9-4\sqrt{5}).
\end{align*}
Applying Theorem \ref{thm:WIN2a} to these gives the following.  For $p \geq 11$ with $\left(\frac2p\right) = 1$, 
\begin{equation}
\label{eq:example4.2b}\left[F_{\frac{12}{3-\sqrt{2}}, \bf{b}}\right]_{p-1}\!(12\sqrt{2}-17) \equiv \sgn\cdot \left(\frac{18-12\sqrt{2}}{p}\right)\cdot p \pmod{p^2}.
\end{equation}
For $p\geq 5$, 
\begin{equation}
    \label{eq:example4.3b}
    \left[F_{\frac{12(15+4\sqrt{5})}{29},{ \bf b}}\right]_{p-1}\!\!(72\sqrt{5}-161) \equiv \sgn \cdot\left(\frac{-10(162-72\sqrt{5})}{p}\right)\cdot p \pmod{p^2}.  
\end{equation}
For $p\neq 29$ with $p\equiv \pm 1 \pmod 5$, 
\begin{equation}
    \label{eq:example4.4b}
\left[F_{6,{\bf b}}\right]_{p-1}\left(\frac{1}{4}\right) \equiv \sgn \cdot \left(\frac{3/4}{p}\right)\cdot p \pmod{p^2}.  
\end{equation}
For $p\geq 11$ with $\left(\frac{5}{p}\right)=1$, 
\begin{equation}
    \label{eq:example4.5b}
        \left[ F_{5+\sqrt{5},{\bf b}}\right]_{p-1} (9-4\sqrt{5}) \equiv \sgn \cdot \left(\frac{4\sqrt{5}-8}{p}\right)\cdot p \pmod{p^2}.
\end{equation}

\section{Connections with modular forms}\label{sec: modular}

\renewcommand\theequation{\arabic{section}.\arabic{equation}}

By work of Freitas, Hung, and Siksek \cite{FLS}, elliptic curves over real quadratic fields are modular.  The symmetric square $\rho_{\l, \ell} := \Sym^2(\rho_{E, \ell})$ of the compatible family of Galois representations $\rho_{E, \ell}$ attached to $E=E_{-}:=\widetilde{E}_d\left(\frac{1-\sqrt{1-\lambda}}{2}\right)$ is also modular. 
When $E$ has CM, $\rho_{\l, \ell}$ can be described in terms of a Grossencharacter attached to $E$ and corresponds to a weight $3$ CM cusp form by Weil's Converse Theorem \cite{Weil, Ogg-book}.  
The following is a corollary to Theorem \ref{thm:WIN2b} using the theory of CM elliptic curves.

\begin{corollary}\label{cor: modularity}
\label{weight3}
Assume the notation and assumptions in Theorem  \ref{thm:WIN2b} with $\l \in \Q$. Then there exists a weight $3$ Hecke eigenform $f$ such that for primes $p$ where $p\nmid 2d$ and $\l$ and $1-\l$ are $p$-adic units,
\[
[F_{0,(\frac12, \frac{1}{d}, \frac{d-1}{d})}]_{p-1}(\l) \equiv a_p(f) \pmod {p^2},
\] 
where $a_p(f)$ is the $p$th Fourier coefficient of $f$. 
\end{corollary}
\noindent We note that the case when $\l=1$ is given by Mortenson \cite{Mortenson}. 

\begin{proof} 
By the modularity of CM elliptic curves over real quadratic fields, there exist Grossencharacters $\psi_+$ and $\psi_-$ such that the $L$-functions of $\psi_+$ and $E_+:=\widetilde{E}_d\left(\frac{1+\sqrt{1-\lambda}}{2}\right)$, and respectively $\psi_-$ and $E_-:=\widetilde{E}_d\left(\frac{1-\sqrt{1-\lambda}}{2}\right)$, coincide. Namely, 
\begin{align*}
L(\psi_+, s) &=L(E_+/\Q(\sqrt{1-\l}),s), \\
L(\psi_-, s) &=L(E_-/\Q(\sqrt{1-\l}),s).   
\end{align*}  
Let $\psi_{k_d}$ be the character induced by the quadratic character $\left(\frac{k_d}\cdot\right)$.  Then there exists a weight $3$ CM modular form $f$ with an imprimitive Dirichlet character $\chi$ determined by the CM discriminant corresponding to $\psi_{k_d}\psi_-\psi_+$.  Hence, by Theorem \ref{thm:WIN2b} one has for supersingular primes  ($\chi(p)=-1$) that
$$
[F_{0,(\frac12, \frac{1}{d}, \frac{d-1}{d})}]_{p-1}(\l) \equiv 0 =a_p(f) \pmod{p^2}, 
$$
and for ordinary primes ($\chi(p)=1$), 
\begin{align*}
[F_{0,(\frac12, \frac{1}{d}, \frac{d-1}{d})}]_{p-1}(\l) & \equiv 
\begin{cases}
u_p^2(E_-)  & \text{if } \left(\frac{1-\l}p \right)=1, \\
\left( \frac{k_d}{p} \right)u_p^2(E_-)  & \text{if } \left(\frac{1-\l}p \right)=-1, 
\end{cases} \\
& \equiv 
\begin{cases}
\left( \frac{k_d}{p} \right)u_p(E_-) u_p(E_+) & \text{if } \left(\frac{1-\l}p \right)=1, \\
\left( \frac{k_d}{p} \right)u_p(E_-) u_p(E_+)   & \text{if } \left(\frac{1-\l}p \right)=-1, 
\end{cases} \\
& \equiv  u_p(f)+  p^2/u_p(f)  \pmod{p^2},
\end{align*}
where $u_p(f)$ and $ p^2/u_p(f)$ are the roots of the Hecke polynomial $T^2-a_p(f)T+p^2$ 
attached to $f$ at $p$. 
\end{proof}

By the above discussion, we have modularity for the Galois representations of $\mbox{Gal}(\ol \Q/\Q)$ attached to $F_{0,(\frac12, \frac{1}{d}, \frac{d-1}{d})}(\l)$ arising from the tensor product of the Galois representations of $\mbox{Gal}(\ol \Q/\Q(\sqrt{1-\l}))$ attached to the CM curves $E_-$ and $E_+$.  
In particular, denote $\rho_\ell:=\rho_{E_-} \otimes \rho_{E_+}$. Then $\rho_\ell$ decomposes as 
$$\rho_\ell = \sigma_\ell \oplus \epsilon_\ell\oplus \epsilon_\ell,$$  where $\epsilon_\ell$ is the cyclotomic character tensored by the CM quadratic character and $\sigma_\ell$ is a 2-dimensional representation satisfying $\text{Tr}(\sigma_\ell(\Frob_p))=a_p(g)$ for a weight-3 form $g$ with CM arising from the Grossencharacter $\psi_-\psi_+$ defined in the proof of Corollary~\ref{cor: modularity}. The relationship between $g$ and the cusp form $f$ in the proof of Corollary \ref{cor: modularity} is $$f=\left(\frac{k_d}\cdot\right)\otimes g,$$  which follows from  Theorem \ref{thm:WIN2b}.   In particular, if $p$ is supersingular, since $g$ has CM it follows that $a_p(g) =0$.  If $p$ is ordinary, then since the eigenvalues of $\rho_{E_-, \ell}$ are $\varepsilon u_p(E_-)$ and $\varepsilon p/u_p(E_-)$, with $\varepsilon^2= 1$, it follows that the eigenvalues of $\rho_\ell$ are $u_p(E_-)^2$, $p$, $p$, and $p^2/u_p(E_-)^2$. Thus,
\[
a_p(g)\equiv  \text{Tr}(\sigma_\ell(\Frob_p)) = u_p(E_-)^2 + \frac{p^2}{u_p(E_-)^2}  \pmod{p^2}.
\]

\subsection{Special values and L-values.} \label{subseq: specialvals} 
In \cite[Exs. 5.2-5.6]{WIN5}, some special values of the weight $4$ modular form $Z(\tau) = F_{0,(\frac12, \frac{1}{d}, \frac{d-1}{d})}(X(\tau))^2$ are computed in terms of Gamma functions.  Here we interpret these in terms of $\Omega_K$, defined in (\ref{omegadef}). As discussed in the proof of Lemma \ref{lem:deltaquotioent}, 
\[
F_{0,(\frac12, \frac1d,\frac{d-1}{d})}(X(\gamma\tau_0))^2 \sim \Omega_K^4,
\]
where $K$ is the CM-field of the elliptic curve $\widetilde{E}_d\left(\frac{1-\sqrt{1-X(\gamma \tau_0)}}{2}\right)$. 
In particular, the evaluations of $F_{0,(\frac12, \frac{1}{d}, \frac{d-1}{d})}(X(\gamma \tau_0))^2$ given in \cite[Exs. 5.2-5.6]{WIN5} yield the following equalities, 

\begin{align}
        F_{0,(\frac 1{12}, \frac{5}{12})}\left(\frac{27}{125}\right)^2 & =  F_{0,(\frac12, \frac{1}{6}, \frac{5}{6})}\left(\frac{27}{125}\right) =\frac{\sqrt 5}8 \Omega_{\Q(\sqrt{-8})}^2 \label{eq:sp_val_1} \\
         F_{0,(\frac 1{12}, \frac{5}{12})}\left(\frac{4}{125}\right)^2& =  F_{0,(\frac12, \frac{1}{6}, \frac{5}{6})}\left(\frac{4}{125}\right)=\frac{\sqrt {15}}{ 3 \cdot \sqrt[3]{32}} \Omega_{\Q(\sqrt{-3})}^2\\
         F_{0,(\frac 1{8}, \frac{3}{8})}\left(\frac{1}{9}\right)^2& =  F_{0,(\frac12, \frac{1}{4}, \frac{3}{4})}\left(\frac{1}{9}\right)=\frac1{\sqrt {24}} \Omega_{\Q(\sqrt{-24})}^2\\
         F_{0,(\frac 1{3}, \frac{1}{6})}\left(\frac{1}{2}\right)^2& =  F_{0,(\frac12, \frac{1}{3}, \frac{2}{3})}\left(\frac{1}{2}\right)=\frac{\sqrt {3}}{8} \Omega_{\Q(\sqrt{-24})}^2\\    
         F_{0,(\frac 1{3}, \frac{1}{6})}\left(\frac{4}{125}\right)^2& =  F_{0,(\frac12, \frac{1}{3}, \frac{2}{3})}\left(\frac{4}{125}\right)=\frac{\sqrt {3}}{8} \Omega_{\Q(\sqrt{-15})}^2. \label{eq:sp_val_5}
    \end{align}

We expect that the special values in \eqref{eq:sp_val_1}-\eqref{eq:sp_val_5} are each related to the $L$-value of a weight $3$ CM modular form through the modularity of CM elliptic curves.    
For example, by \cite[Theorem 2]{LLT} we have 
$$
     \pi^2\cdot F_{0,(\frac12, \frac{1}{6}, \frac{5}{6})}\left(\frac{4}{125}\right)     =6\sqrt 5 \cdot  L(\eta(2\tau)^3\eta(6\tau)^3,2)
     =  9\sqrt 5 \cdot  L(\eta(6\tau)^4,1)^2,
$$
where for the elliptic curve $E:y^2=x^3+1$,
$$ L(\eta(6\tau)^4,1) =L(E,1)=\frac{\pi}{3\sqrt[4]{3}\sqrt[6]{32}}\Omega_{\Q(\sqrt{-3})}.$$ 
In particular, the following relation between weight-3 cusp forms, $f_{300.3.g.b}= \left(\frac 5 \cdot \right) \otimes\eta(2\tau)^3\eta(6\tau)^3$, provides the algebraic relation 
$$
L(f_{300.3.g.b},2)\sim \sqrt{5} L(\eta(2\tau)^3\eta(6\tau)^3,2), 
$$
where $f_{300.3.g.b}$ is the weight-3 cusp form identified in Corollary \ref{cor: modularity}.  We expect that the methods in \cite{LLT} could be used to show that $$\pi^2\cdot F_{0,(\frac12, \frac{1}{6}, \frac{5}{6})}\left(\frac{4}{125}\right) =\frac{25}4 L(f_{300.3.g.b},2).$$  

\noindent In light of this, We make the following conjecture based on numeric evidence. 

\begin{conjecture}
For modular forms $f$ labeled with their LMFDB labels \cite{LMFDB},
\begin{align*}
        \pi^2\cdot F_{0,(\frac12, \frac{1}{6}, \frac{5}{6})}\left(\frac{27}{125}\right)& =\frac{25}2 L(f_{\href{https://www.lmfdb.org/ModularForm/GL2/Q/holomorphic/800/3/g/a/}{800.3.g.a}},2), \\
         \pi^2\cdot  F_{0,(\frac12, \frac{1}{6}, \frac{5}{6})}\left(\frac{4}{125}\right)&= \frac{25}4 L(f_{\href{https://www.lmfdb.org/ModularForm/GL2/Q/holomorphic/300/3/g/b/}{300.3.g.b}},2),\\
        \pi^2\cdot   F_{0,(\frac12, \frac{1}{4}, \frac{3}{4})}\left(\frac{1}{9}\right)&= 12L( f_{\href{https://www.lmfdb.org/ModularForm/GL2/Q/holomorphic/24/3/h/a/}{24.3.h.a}},2), \\
       \pi^2\cdot    F_{0,(\frac12, \frac{1}{3}, \frac{2}{3})}\left(\frac{1}{2}\right) &= 9 L( f_{\href{https://www.lmfdb.org/ModularForm/GL2/Q/holomorphic/24/3/h/b/}{24.3.h.b}},2), \\
        \pi^2\cdot    F_{0,(\frac12, \frac{1}{3}, \frac{2}{3})}\left(\frac{4}{125}\right)&=\frac{45}4 L ( f_{\href{https://www.lmfdb.org/ModularForm/GL2/Q/holomorphic/15/3/d/b/}{15.3.d.b}},2).
\end{align*}
\end{conjecture}

We note that for some hypergeometric evaluations at $\l=\pm1$, special $L$-values of the corresponding cusp forms at $1$ and $2$ are obtained in work of Zagier \cite{Zagier}, Osburn and Straub \cite{Osburn-Straub}, as well as Allen, Grove, Long, and the fifth author \cite[Appendix]{HMM2}.

\section{Results and conjectures modulo \texorpdfstring{$p^3$}{}}\label{sec: modp3}

Here we discuss what is known and conjectured modulo $p^3$.  As described in \S \ref{CMcurves}, Zudilin conjectured that Theorem \ref{thm:WIN2a} holds modulo $p^3$.  Some cases of this have been confirmed, for example in works of Beukers \cite{Beukers-supercongruence}, Mortenson \cite{Mortenson-padic}, and Zudilin \cite{Zudilin-supercongruence}.  Computationally Theorem~\ref{thm:WIN2b} appears to hold modulo $p^3$ for ordinary primes, but not for supersingular primes in general.  Some progress on this has been made, including by Long and Ramakrishna \cite{LongRamakrishna} as well as Beukers \cite{Beukers-supercongruence}.  In particular, Beukers \cite[Thm. 1.27]{Beukers-supercongruence} gives the cases when $d=6$ through properties of corresponding modular forms which yields the result modulo $p^3$ given below. 

Let $D$ be a positive integer with $D\equiv 0$ or $3 \pmod4$, and define
\begin{equation}\label{def:omegaD}
\omega_D= \begin{cases} \frac{\sqrt{-D}}2 & \text{ if } D\equiv 0\pmod 4, \\ \frac{1+\sqrt{-D}}2 & \text{ if } D\equiv 3\pmod 4. \end{cases}
\end{equation}

\begin{theorem}(Theorems \ref{thm:WIN2a}, \ref{thm:WIN2b}, and \cite[Theorem 1.27]{Beukers-supercongruence})\label{thm:Beukers}
Assume the notation and assumptions in Theorem \ref{thm:WIN2a} and \eqref{def:omegaD}. Suppose $\l=1728/j(\omega_D)$.  Then, for an ordinary prime $p$, 
\[
[F_{\alpha,\left(\frac{1}{2},\frac{1}{6},\frac56\right)}]_{p-1}(\lambda)\equiv   \left( \frac{1-\l}p \right) \cdot p \pmod {p^3},
\]
and 
\[
[F_{0,(\frac12, \frac{1}{6}, \frac{5}{6})}]_{p-1}(\l) \equiv
\begin{cases}
u_p^2  & \text{if }  \left(\frac{1-\l}p \right)=1, \\
\left( \frac{-1}{p} \right)u_p^2  & \text{if } \left(\frac{1-\l}p \right)=-1, 
\end{cases} \pmod {p^3},
\]
where  $u_p$ is a fixed $p$-adic unit root of the geometric Frobenius at $p$ acting on the first cohomology of the elliptic curve $\widetilde{E}_d(\frac{1-\sqrt{1-\lambda}}{2})$.  
\end{theorem}
One may notice that the left-hand sides of the congruences in Theorem \ref{thm:Beukers} are similar to the left-hand sides of \eqref{eq:example5.2} and \eqref{eq:example5.3}. In fact, as we will demonstrate, \eqref{eq:example5.2} and \eqref{eq:example5.3} can be partially extended to hold modulo $p^3$, and we conjecture that a special case of all supercongruences arising from Theorem \ref{thm:main} hold modulo $p^3$. 

\renewcommand\theequation{\Alph{equation}}
\setcounter{equation}{15}

\begin{conjecture}\label{conj:Product}
Assume the hypotheses in Theorem~\ref{thm:main} for ${\bf b} = (\frac12, \frac{1}{d}, \frac{d-1}{d})$. Then for primes $p\geq 5$, 
\begin{multline*}
[F_{0,{\bf b}} \cdot F_{2\alpha_N, {\bf b}}]_{p-1}(\lambda) \equiv \\ 
\begin{cases}   
0  &  \text{if } \widetilde{E}_d\left(\frac{1-\sqrt{1-\lambda}}{2}\right) \text{ is supersingular at } p\\
 u_p^2 \, p   &  \text{if } \widetilde{E}_d\left(\frac{1-\sqrt{1-\lambda}}{2}\right) \text{ is ordinary at } p \text{ and }\left(\frac{1-\lambda}p \right)=1 \\
-\left( \frac{k_d}{p}\right)u_p^2 \, p &  \text{if } \widetilde{E}_d\left(\frac{1-\sqrt{1-\lambda}}{2}\right) \text{ is ordinary at } p \text{ and }\left(\frac{1-\lambda}p \right)=-1 
\end{cases} \pmod{p^3},
\end{multline*}
where $u_p$ is the unit root of the geometric Frobenius at $p$ acting on the first cohomology of the elliptic curve $\widetilde{E}_d(\frac{1-\sqrt{1-\lambda}}{2})$ and $k_2=k_6=-1$, $k_3=-3$, $k_4=-2$.
\end{conjecture}

The cases of Conjecture \ref{conj:Product} arising from \eqref{eq:example5.2} and \eqref{eq:example5.3} have been verified using Theorem \ref{thm:Beukers}. In particular, \eqref{eq:example5.2} can be extended to obtain for $\l=\frac{27}{125}$ and primes $p\neq 2,5$,
\setcounter{equation}{15}
\renewcommand\theequation{\Alph{equation}}
  \begin{multline} 
  \label{eq:example5.2b}
 \left[F_{0,(\frac12, \frac{1}{6}, \frac{5}{6})}\cdot F_{\frac{28}{3},(\frac12, \frac{1}{6}, \frac{5}{6})}\right]_{p-1}\left(\frac{27}{125}\right) \\ \equiv 
\begin{dcases}
pu_p^2 & \text{ if } \left( \frac{-2}{p} \right) =1 \text{ and } \left( \frac{1-\lambda}{p} \right) =1\\
-\left(\frac{-1}{p}\right)pu_p^2 & \text{ if } \left( \frac{-2}{p} \right) =1 \text{ and } \left( \frac{1-\lambda}{p} \right) =-1\\
\end{dcases}
\pmod{p^3}.
  \end{multline}
Similarly, we can extend \eqref{eq:example5.3} to obtain for primes $p\geq 7$,
\begin{multline}
\label{eq:example5.3b}
   \left[F_{0,(\frac12, \frac{1}{6}, \frac{5}{6})}\cdot F_{11,(\frac12, \frac{1}{6}, \frac{5}{6})}\right]_{p-1}\left(\frac{4}{125}\right) \\ \equiv 
\begin{dcases}
pu_p^2 & \text{ if } \left( \frac{-3}{p} \right) =1 \text{ and } \left( \frac{5}{p} \right)=1\\
-\left(\frac{-1}{p}\right)pu_p^2 & \text{ if } \left( \frac{-3}{p} \right) =1 \text{ and } \left( \frac{5}{p} \right)=-1\\
\end{dcases}
\pmod{p^3}.
\end{multline}

\setcounter{equation}{6}
\renewcommand{\theequation}{\arabic{section}.\arabic{equation}}

We now show that Zudilin's conjecture implies Conjecture \ref{conj:Product}.  

\begin{proposition}\label{prop:p^3}
If Theorem \ref{thm:WIN2a} holds modulo $p^3$, then Conjecture \ref{conj:Product} is true. 
\end{proposition}

\begin{proof} 
If Theorem \ref{thm:WIN2a} holds modulo $p^3$, then using Lemma \ref{lem:errmodp^3} we have
\begin{multline*}
[F_{0,{\bf b}}\cdot F_{2\alpha_N, {\bf b}}]_{p-1}(\l) \equiv \\
\begin{cases}
-\left( \frac{1-\l}p \right) [F_{0,{\bf b}}]_{p-1}(\l) \cdot p  & \text{if } \widetilde{E}_d\left(\frac{1-\sqrt{1-\lambda}}{2}\right) \text{ is supersingular at } p\\
[F_{0,{\bf b}}]_{p-1}(\l) \cdot p  & \text{if } \widetilde{E}_d\left(\frac{1-\sqrt{1-\lambda}}{2}\right) \text{ is ordinary at } p \text{ and }\left(\frac{1-\l}p \right)=1 \\
-[F_{0,{\bf b}}]_{p-1}(\l) \cdot p  & \text{if } \widetilde{E}_d\left(\frac{1-\sqrt{1-\lambda}}{2}\right) \text{ is ordinary at } p \text{ and }\left(\frac{1-\l}p \right)=-1 
\end{cases}
\pmod {p^3}.
\end{multline*}
Thus applying Theorem \ref{thm:WIN2b} to $[F_{0,{\bf b}}]_{p-1}(\l) \cdot p$ to obtain congruences modulo $p^3$, we obtain the result.
\end{proof}
We numerically verified for applicable primes $p<150$ that Zudulin's conjecture \cite{Zudilin} extending Theorem \ref{thm:WIN2a} to hold modulo $p^3$ holds for the following congruences stemming from \eqref{eq:example5.4}-\eqref{eq:example5.6},
\begin{align}
   \left[ F_{8,(\frac12, \frac{1}{4}, \frac{3}{4})}\right]_{p-1}\left(\frac{1}{9}\right) \equiv  {\sgn} \cdot \left( \frac{1-\l}p \right) \cdot p & \pmod {p^3}, \label{5.9p^3} \\
   & \text{for $\lambda=\frac{1}{9}$ and  primes $5\le p<150$;} \nonumber 
\end{align}
\begin{align}   
    \left[F_{6,(\frac12, \frac{1}{3}, \frac{2}{3})}\right]_{p-1}\left(\frac{1}{2}\right) \equiv  {\sgn} \cdot \left( \frac{1-\l}p \right) \cdot p & \pmod {p^3}, \label{5.10p^3} \\
    & \text{for $\lambda=\frac{1}{2}$ and  primes $5\le p<150$;} \nonumber \\
    \left[F_{\frac{33}{4},(\frac12, \frac{1}{3}, \frac{2}{3})}\right]_{p-1} \left(\frac{4}{125}\right)\equiv  {\sgn} \cdot \left( \frac{1-\l}p \right) \cdot p & \pmod {p^3}, \label{5.11p^3} \\
    & \text{for $\lambda=\frac{4}{125}$ and  primes $7\le p <150$, $p\ne 11$.} \nonumber 
\end{align}
By Proposition \ref{prop:p^3}, if \eqref{5.9p^3}, \eqref{5.10p^3}, or \eqref{5.11p^3} hold, then \eqref{eq:example5.4}, \eqref{eq:example5.5}, and \eqref{eq:example5.6}, respectively hold modulo $p^3$.  However, we note that Theorem~\ref{thm:WIN2b} only seems to hold modulo $p^3$ for ordinary primes. For supersingular primes, the result only holds modulo $p^2$.  We have checked this for the $F_{0,{\bf b}}(z)$ term corresponding to \eqref{eq:example5.4}, \eqref{eq:example5.5}, and \eqref{eq:example5.6}, for the appropriate primes $p<150$.

Finally,  when ${\bf b} = (\frac{1}{d}, \frac{d-1}{d})$  in Theorem \ref{thm:main}, the congruences do not seem to generally extend to a higher power of $p$. For instance, \eqref{eq:example4.2a},   \eqref{eq:example4.3a},  \eqref{eq:example4.4a},  \eqref{eq:example4.5a},  \eqref{eq:example4.6} do not  hold modulo $p^3$ for appropriate primes $p<150$. However, for appropriate primes $p<150$ Example \eqref{eq:example4.7} does hold modulo $p^3$.

We next make the following conjecture modulo $p^3$ related to Corollary \ref{cor: modularity}.
\begin{conjecture}\label{conj:mod3}
Assume the hypotheses in Corollary \ref{cor: modularity}.  For each $[F_{0,(\frac12, \frac{1}{d}, \frac{d-1}{d})}]_{p-1}(\l)$ appearing in  Corollary \ref{weight3} such that the corresponding $a_p(f) \neq 0$, 
$$
[F_{0,(\frac12, \frac{1}{d}, \frac{d-1}{d})}]_{p-1}(\l) \equiv u_p(f)  \pmod{p^3}, \\
$$ 
where $u_p(f) $ is the unit root of the Hecke polynomial $T^2- a_p(f)T+p^2$ associated to $f$. 
\end{conjecture}

\section{Appendix --- Proof of Theorem \ref{thm:WIN2b} and hypergeoemtric character sums} 
\label{sec:appendix}
\renewcommand\theequation{\arabic{equation}}
\renewcommand{\theequation}{\arabic{section}.\arabic{equation}}

In this appendix, we provide an explanation of  Theorem \ref{thm:WIN2b} through the arithmetic properties of the hypergeometric elliptic families, and further record some parallel results in terms of finite hypergeometric functions introduced by Beukers, Cohen, and Mellit in \cite{BCM}, which allow us to see the supercongrunce in Theorem \ref{thm:WIN2b} from the character sum perspective.

\subsection{Proof of Theorem \ref{thm:WIN2b}}  \label{sec: WIN2proof} 

When $E$ is an elliptic curve defined over a finite field $\F_q$ containing $q$ elements, define 
\begin{equation*} 
a_q(E) := q+1-\#  E(\F_q).
\end{equation*}

\begin{proposition}\label{prop: traces} 
    Let $d\in \{2,3,4,6\}$, $p\geq 5$ prime, and $\F_q$ a finite field with $q=p^r$ elements. For $t\in \F_q$ such that both $\widetilde E_{d}(t)$ and  $\widetilde E_{d}(1-t)$ are elliptic curves defined over $\F_q$, one has
    $$
      a_q(\widetilde E_{d}(t)) =\phi_q(k_d)a_q(\widetilde E_{d}(1-t)),
    $$
    where $\phi_q$ is the quadratic character of $\F_q$, and $k_2=k_6=-1$, $k_3=-3$, $k_4=-2$. 
\end{proposition}
We believe this proposition has been documented in literature, however we find the following approach enlightening as we provide a detailed description of isogenies among the elliptic curves. One may see \cite[Lemma 4]{HLLT} for a hypergeometric approach.  The way we approach this proposition is through the moduli interpretation of elliptic curves as complex tori, and the action of $W$ provided in Table \ref{fams}.  The action of the normalizer $$W: t(W\tau) = 1-t(\tau)$$ gives an isogeny from the complex torus $\C/{\Z+\Z\tau}$ to $\C/{\Z+\Z(W\tau)}$ and the composition of this isogeny and its dual gives $W^2=-k_d$-multiplication.  These isogenies induce the isogenies on the reduction of the given elliptic curves.  For each case, the isogeny is a composition of a isogeny defined over $\F_q$ and a quadratic twist.  In the proof of Proposition \ref{prop: traces} below, we provide the explicit isogenies and quadratic twists. 

\begin{proof}[Proof of Prop. \ref{prop: traces}]
Fix $t\in \F_q$, $d\in \{2,3,4,6\}$ such that $\widetilde E_{d}(t)$ and  $\widetilde E_{d}(1-t)$ are defined over $\F_q$.  
 When $d=2$, 
  $\widetilde E_{2}(t)$ is isomorphic to  $\widetilde E_{2}'(t): \, -y^2=x(1-x)(x-(1-t))$ by the map
  $$
    (x,y) \mapsto \left(-x+1, y \right).
  $$
  Therefore, 
  \begin{align*}
     a_q(\widetilde E_{2}(t)) =&a_q(\widetilde E_{2}'(t))=q+1- \#(\widetilde E_{2}'(t)(\F_q) \\
     =& \sum_{x \in \F_q} \phi_q(-x(1-x)(x-(1-t))) =  \phi_q(-1)\sum_{x \in \F_p} \phi_q(x(1-x)(x-(1-t)))\\
     =&    \phi_q(-1) a_q(\widetilde E_{2}(1-t)). 
  \end{align*}

  When $d=6$, $\widetilde E_{d}(t)$ is isomorphic to  $\widetilde E_{d}'(t): \, y^2=x^3-27x+54(1-2t)$ by the map
  $$
    (x,y) \mapsto \left(6^2x+3, 6^3y+3\cdot 6^2x \right),
  $$ 
  and the curve $-y^2=x^3-27x+54(1-2t)$ is isomorphic to $\widetilde E_{d}'(1-t) : y^2=x^3-27x-54(1-2t)$ via $(x,y)\mapsto (-x,-y)$.   This gives 
   $$
      a_q(\widetilde E_{6}(t)) =\phi_q(-1)a_q(\widetilde E_{6}(1-t)). 
    $$

  When $d=4$,  
  $\widetilde E_{4}(t)$ is  2-isogenous to the curve  $\widetilde E_{4}'(t): \, y^2=x(x^2-2x+(1-t))$ by the map
  $$
    (x,y) \mapsto \left(x+1+\frac t{4x}, y(1-\frac t{4x^2}) \right),
  $$
   and the quadratic twist of the curve  $\widetilde E_{4}'(t)$ by $-2$ is isomorphic to  $\widetilde E_{4}(1-t)$ by the 
   isomorphism from $E: -2y^2=x(x^2-2x+(1-t)) \longrightarrow\widetilde E_{d}(1-t)$ given by 
   \[
   (x,y) \mapsto \left(- x/2,y/2\right). 
   \]

  Putting this all together, we get 
  \begin{align*}
     a_q(\widetilde E_{d}(t)) =&a_q(\widetilde E_{d}'(t))=q+1- \#(\widetilde E_{d}'(t)(\F_q) \\
     =& \sum_{x \in \F_q} \phi_q(x(x^2-2x+(1-t))) =  \phi_q(-2)\sum_{x \in \F_p} \phi_q(-2x(x^2-2x+(1-t)))\\
     =&    \phi_q(-2) a_q(\widetilde E_{d}(1-t)). 
  \end{align*}

 Similarly, in the case of $d=3$, the curve 
  $\widetilde E_{3}(t)$ has a Weierstrass equation $y^2=x^3+\frac 14\left(x+\frac t{27}\right)^2$ through the map $(x,y)\mapsto \left(x, y+\frac12(x+\frac t{27})\right)$, which  is $3$-isogenous to $\widetilde E_{3}'(t): \, y^2=x^3-\frac 34(x-\frac {1-t}{9})^2$. The curve   $\widetilde E_{3}(1-t)$ is isomorphic to the quadratic twist of  $\widetilde E_{3}'(t)$ by $-3$ via the isomorphism from $E: -3y^2=x^3-\frac 34\left(x-\frac {1-t}{9}\right)^2  \longrightarrow\widetilde E_{3}(1-t)$ given by
\[
(x,y) \mapsto \left(- x/3,y/3\right). 
\]
 This gives the desired identity
$$
      a_q(\widetilde E_{3}(t)) =\phi_q(-3)a_q(\widetilde E_{3}(1-t)). 
$$
\end{proof}

For our next lemma, we $p$-adically relate the evaluations of the truncation of $F_{0,(\frac{1}{d},\frac{d-1}{d})}$ at $\l$ and $1-\l$, respectively.

\begin{lemma} \label{lemma: p-Clausen} Let $d\in \{2,3,4,6\}$, $p\geq 5$ prime, and $k_2=k_6=-1$, $k_3=-3$, $k_4=-2$.  Then for $\l\in \ol \Q$ such that $\l$ can be embedded in $\Z_p$,
$$
    [F_{0,(\frac{1}{d},\frac{d-1}{d})}]_{p-1}(\l)  \equiv  \left(\frac{k_d}p\right) [F_{0,(\frac{1}{d},\frac{d-1}{d})}]_{p-1}(1-\l)  \pmod{p^2}.
$$
Furthermore, as polynomials in $t$,
$$
[F_{0,(\frac{1}{d},\frac{d-1}{d})}]_{p-1}(t)  - \left(\frac{k_d}p\right) [F_{0,(\frac{1}{d},\frac{d-1}{d})}]_{p-1}(1-t) \in p^2\Z_p[t],
$$
$$
[F_{0,(\frac12,\frac{1}{d},\frac{d-1}{d})}]_{p-1}(t) - \left(\frac{k_d}p\right)\left[F_{0,(\frac{1}{d},\frac{d-1}{d})}\right]_{p-1}\left(\frac{1+ \sqrt{1-t}}{2}\right) \cdot \left[F_{0,(\frac{1}{d},\frac{d-1}{d})}\right]_{p-1}\left(\frac{1- \sqrt{1-t}}{2}\right) \in p^2\Z_p[t].
$$
\end{lemma}

\begin{proof} 
Trivially, $[F_{0,(\frac{1}{d},\frac{d-1}{d})}]_{p-1}(0)=1$. Moreover, one can show $[F_{0,(\frac{1}{d},\frac{d-1}{d})}]_{p-1}(1)\equiv \left(\frac{k_d}p\right)\pmod p$ by first using \eqref{lem:win2} and the Chu-Vandermonde identity to obtain 
$$[F_{0,(\frac{1}{d},\frac{d-1}{d})}]_{p-1}(1) \equiv [F_{0,(\frac{1}{d},\frac{d-1}{d})}]_{r_d}(1) = \frac{ \left( \frac{1}{d} \right)_{r_d}}{\left( 1 \right)_{r_d}} \pmod{p}, $$ 
where $r_d\equiv -1/d \pmod{p}$ is defined in \eqref{def:rdsd}, then interpreting the right hand side in terms of $p$-adic Gamma functions and  \cite[Thm. 14]{LongRamakrishna}, then using quadratic reciprocity.  From  counting points on  classical hypergeometric elliptic families (see  \cite{HLLT, Long18, Beukers-Stienstra} for example) it follows from Proposition \ref{prop: traces} that as polynomials in $\Z_p[t]$,
\begin{equation} \label{eq:HLLT}
    [F_{0,(\frac{1}{d},\frac{d-1}{d})}]_{p-1}(t)  \equiv  \left(\frac{k_d}p\right) [F_{0,(\frac{1}{d},\frac{d-1}{d})}]_{p-1}(1-t)  \pmod p.
\end{equation}
Thus there exists $n(t)\in \Z_p[t]$ such that 
\begin{equation} \label{eq:pmult}
[F_{0,(\frac{1}{d},\frac{d-1}{d})}]_{p-1}(t) = \left(\frac{k_d}p\right) [F_{0,(\frac{1}{d},\frac{d-1}{d})}]_{p-1}(1-t) + pn(t),
\end{equation}
and squaring both sides shows that 
\begin{equation} \label{eq:2F1relation}
[F_{0,(\frac{1}{d},\frac{d-1}{d})}]_{p-1}^2(t) \equiv [F_{0,(\frac{1}{d},\frac{d-1}{d})}]_{p-1}^2(1-t) + 2p\left(\frac{k_d}p\right) n(t) [F_{0,(\frac{1}{d},\frac{d-1}{d})}]_{p-1}(1-t) \pmod{p^2}.
\end{equation}

From \cite[Lemma 18]{WIN2} and the fact that $4t(1-t)$ is invariant under the change of variable $t\mapsto 1-t$, it follows that
\begin{equation} \label{eq:WIN2Lem18}
  [F_{0,(\frac12,\frac{1}{d},\frac{d-1}{d})})]_{p-1}(4t(1-t)) \equiv [F_{0,(\frac{1}{d},\frac{d-1}{d})}]_{p-1}^2 (t) \equiv [F_{0,(\frac{1}{d},\frac{d-1}{d})}]_{p-1} ^2 (1-t) \pmod{p^2}.
\end{equation}
We conclude from \eqref{eq:WIN2Lem18} and \eqref{eq:2F1relation} that as polynomials in $t$, $n(t)[F_{0,(\frac{1}{d},\frac{d-1}{d})}]_{p-1}(1-t) \equiv 0 \pmod{p}$.  However, clearly $[F_{0,(\frac{1}{d},\frac{d-1}{d})}]_{p-1}(t) \not \equiv 0 \pmod p$ as it has constant term $1$.  Thus, $n(t) \equiv 0 \pmod{p}$, so \eqref{eq:pmult} gives that
\begin{equation} \label{eq:t(1-t)}
[F_{0,(\frac{1}{d},\frac{d-1}{d})}]_{p-1}(t)  \equiv  \left(\frac{k_d}p\right) [F_{0,(\frac{1}{d},\frac{d-1}{d})}]_{p-1}(1-t) \pmod{p^2}, 
\end{equation}
which yields the first and second statements.  Furthermore, \eqref{eq:t(1-t)} and \eqref{eq:WIN2Lem18} give that   
\begin{equation}\label{eqn: truncated Clausen}
  [F_{0,(\frac12,\frac{1}{d},\frac{d-1}{d})}]_{p-1}(4t(1-t)) \equiv   \left(\frac{k_d}p\right)  [F_{0,(\frac{1}{d},\frac{d-1}{d})}]_{p-1}(t) \cdot  [F_{0,(\frac{1}{d},\frac{d-1}{d})}]_{p-1}(1-t) \pmod{p^2}. 
\end{equation}
From here it remains to substitute $\frac{1+ \sqrt{1-t}}{2}$ for $t$, but we need to establish that we obtain a polynomial in only $t$.  This follows from the fact that
$$
  F(u,v):= [F_{0,(\frac{1}{d},\frac{d-1}{d})}]_{p-1}(u) \cdot [F_{0,(\frac{1}{d},\frac{d-1}{d})}]_{p-1}(v)
$$
is a symmetric polynomial in $\Z_p[u,v]$, and hence it can written as a polynomial in the elementary symmetric polynomials $u+v, uv$.  Thus setting $u=\frac{1+ \sqrt{1-t}}{2}$ and $v=\frac{1- \sqrt{1-t}}{2}$, there exists $G(x,y) \in \Z_p[x,y]$ such that $F(u,v)=G(u+v,uv)=G(1, t/4) \in \Z_p[t]$. Thus we can substitute $\frac{1+ \sqrt{1-t}}{2}$ for $t$ in \eqref{eqn: truncated Clausen} to conclude
$$
[F_{0,(\frac12,\frac{1}{d},\frac{d-1}{d})}]_{p-1}(t) - \left(\frac{k_d}p\right) [F_{0,(\frac{1}{d},\frac{d-1}{d})}]_{p-1}\left(\frac{1+ \sqrt{1-t}}{2}\right) \cdot [F_{0,(\frac{1}{d},\frac{d-1}{d})}]_{p-1}\left(\frac{1- \sqrt{1-t}}{2}\right) \in p^2\Z_p[t]. 
$$
\end{proof}

We now give a proof of Theorem \ref{thm:WIN2b} for the ordinary cases. Some useful information and terminology can be found in  \cite[Chapter 2]{Silverman-adv} and \cite[Chapter 6]{Katz-crystalline} for example.

\begin{proof}[Proof of Theorem \ref{thm:WIN2b} for ordinary primes $p$.] 
Let $\l$ be a totally real number satisfying the assumptions of Theorem \ref{thm:WIN2b}.  Set $\mu=v_{-}:=\frac{1-\sqrt{1-\l}}2$, $v_{+}:=\frac{1+\sqrt{1-\l}}2$, as well as $E_-:=\widetilde E_{d}(v_-)=\widetilde E_{d}(\mu)$ and $E_+:=\widetilde E_{d}(v_+)$.  For a given prime $p$ such that $\widetilde E_{d}(\mu)$  has good reduction and is ordinary at $p$, let $\mathfrak p$ be a prime ideal in the ring of integers of $\Q(\mu)$ lying above $p$, and $\Q_{\mathfrak p} (\mu)$ be the completion of $\Q(\mu)$ at the place $\mathfrak p$. The ring of integers of $\Q_{\mathfrak p} (\mu)$ is isomorphic to the ring $A$  of Witt vectors of $\F_p(\sqrt{1-\l})$.   In  this proof, we view $E_-$ and $E_+$ as elliptic curves over $A$.

We examine cases based on whether the prime $p$ splits in $\Q(\mu)$, i.e.~when $\left(\frac{1-\l}{p}\right)=1$, or $p$ is inert in $\Q(\mu)$, i.e.~when  $\left(\frac{1-\l}{p}\right)=-1$.

\textit{Case 1}: $\left(\frac{1-\l}{p}\right)=1$. In this case $\Q_{\mathfrak p} (\mu)=\Q_{\mathfrak p}(\sqrt{1-\l})\simeq \Q_p$ 
and we consider $E_-$ as an elliptic curve over $A=W(\F_p) \simeq \Z_p$. Recall that $u_p$ and $p/u_p$ are roots of $T^2-a_{p}(E_-)T +p$, where $a_p(E_-)$ is defined in Proposition \ref{prop: traces}. The endomorphism ring $\mbox{End}(E_-)$ embeds in in $\Z_p$, and since $p$ splits in this case, $\sqrt{1-\l}$ is also embedded in $\Z_p$.  We set $\pi^\ast$ to represent the image of  $p/u_p$ in $\Z_p$. 

Multiplication on $E_-$ by $\pi^\ast$  is a rational function of the local uniformizer $\xi:= -\frac{x}{y}$, and there is an algebraic number $\Delta:=\Delta_{d,\pi^\ast}$ with $\Delta^2 \equiv 1 \pmod{pA}$ such that $[\pi^\ast](\xi)\equiv \Delta \xi^p\pmod{pA}$. This gives rise to a degree-$p$ Frobenius map $\Phi$ on $A$, which induces a morphism 
\begin{align*}
    \widetilde{\Phi}^\ast: \,  H^1_{DR}(\widehat E_-/A, (p)) & \longrightarrow    H^1_{DR}(\widehat E_-/A, (p)) \\
    \sum_{n\geq 1} \frac{c(n-1)}n \xi^n & \mapsto   \sum_{n\geq 1}\frac{c(n-1)^\sigma}n \Phi(\xi)^n =\sum_{n\geq 1} \frac{c(n-1)}n \Delta^n\xi^{pn},
\end{align*}
where $\widehat E_-$ is defined as in Katz \cite{Katz-crystalline}.

The characteristic polynomial of $\widetilde{\Phi}^\ast$ is $T^2-a_p(E_-)T+p$ which has $\pi^*$ as a root, and thus $\pi^*$ is one of the eigenvalues of $\widetilde{\Phi}^\ast$. Analogous to the construction in Section \ref{CMcurves}, the induced action of $\End(E_-)$ on the de Rham space $H^1_{DR}(E_-/A, (p))$ has two common eigenvectors, one of which corresponds to the class of the standard invariant holomorphic differential one form $\omega$, which as in \eqref{def:wvInt} can be represented via $\omega =\sum_n \frac{a(n-1)}n \xi^n$.  By the Cayley-Hamilton Theorem, $\widetilde{\Phi}^\ast(\omega)= \pi^\ast \omega$.  Comparing the coefficients of $\xi^p$ from the left and the right hand sides of $\sum_n \frac{a(n-1)}{n}\Delta^n\xi^{pn} = \pi^* \sum_n \frac{a(n-1)}{n}\xi^{n}$, we get $$\Delta\equiv \pi^\ast \frac{a(p-1)}{p} \pmod {pA},$$
which gives 
$$
a(p-1)\equiv\Delta u_p  \pmod {pA}   \equiv \Delta u_p  \pmod {p\Z_p}.
$$
According to \cite{Coster-Van, WIN2, Beukers-Stienstra}, we have 
\begin{equation}
\label{eq:upcong}
   [F_{0,{\frac 1d,\frac{d-1}d}}]_{p-1}(v_-) \equiv \Delta u_p \pmod{p^2}. 
\end{equation}

Since we are assuming $\left(\frac{1-\l}{p}\right)=1$,  we get from Lemma \ref{lemma: p-Clausen} and \eqref{eq:t(1-t)}  that
\begin{equation}
\label{eq:3F2cong}
[F_{0,(\frac12, \frac{1}{d}, \frac{d-1}{d})}]_{p-1} (\l)
\equiv [F_{0,{\frac 1d,\frac{d-1}d}}]_{p-1}^2(v_-) \equiv \Delta^2 u_p \equiv u_p^2 \pmod {p^2},
\end{equation}
which proves the claim in this case.

\textit{Case 2}: $\left(\frac{1-\l}{p}\right)=-1$.  In this case $\Q_{\mathfrak p} (\mu)=\Q_{\mathfrak p}(\sqrt{1-\l})\simeq \Q_p(\zeta_{p^2-1})$, which is the unique unramified quadratic extension of $\Q_p$, and we view $E_-$ and $E_+$ as elliptic curves over $A=W(\F_p) \simeq \Z_p[\sqrt{1-\lambda}]$. As before, the endomorphism ring $\mbox{End}(E_-)$ is embedded in $\Z_p \subset\,\Z_p[\sqrt{1-\lambda}]$, however in this case, $\sqrt{1-\l}$ can not be embedded in $\Z_p$.  Let $\sigma \in \mbox{Aut}(A)$ be the automorphism of $A$ lifted from the Frobenius automorphism $z\mapsto z^p$ on $\F_p$.  Then, $a^\sigma := \sigma(a)\equiv a^p \pmod{pA}$ for all $a\in A$.  

Analogous to Case 1, $u_p^2$ and $p^2/u_p^2$ are roots of $T^2-a_{p^2}(E_-)T+p^2$ and belong to the CM field $K$, whose completion at $p$ is $\Q_p$.  Denoting from now on $u_p$ and $p/u_p$ as the embedding of these elements in $\Q_p(\sqrt{1-\l})$, we have $u_p^2 \in \Z_p$ and moreover, $u_p^2 \equiv a_{p^2}(E_-) \pmod{p\Z_p}$. Furthermore, since $p$ is inert in this case, $u_p$ is a $p$-adic unit root of the geometric Frobenius associated to  both  $E_-$  as well as $E_+$.  

Note that the  multiplication map $[p^2/u_p^2]$ on $E_-$ or  $E_+$ gives rise to a degree-$p^2$ Frobenius map for which there exists a degree-$p$ map $[\pi^\ast]$ satisfying $[\pi^\ast]^2=[p^2/u_p^2]$.  Then $[\pi^*]$, which is an isogeny from $E_+\rightarrow E_-$ and from $E_-\rightarrow E_+$, serves as a rational function of $\xi$ as stated in \cite{WIN2} (see \S 4.2 and \S 5).  Moreover, associated to each of the curves $E_+$ and $E_-$ there is an algebraic number $\Delta_\pm:=\Delta_{d,\pi^\ast}$ with $\Delta_\pm^4 \equiv 1 \pmod{pA}$ and $\Delta_-\Delta_+\equiv 1 \pmod{p A}$ such that on $E_-$, $[\pi^\ast](\xi)\equiv \Delta_- \xi^p\pmod{p A}$ and on $E_+$, $[\pi^\ast](\xi)\equiv \Delta_+ \xi^p\pmod{p A}$. This gives rise to a degree-$p$ map $\Phi$ on $A$ that induces morphisms
$$
   \,  H^1_{DR}(\widehat E_-/A, (p)) \overset{\widetilde{\Phi}^\ast}{\longrightarrow}    H^1_{DR}(\widehat E_+/A, (p))\overset{\widetilde{\Phi}^\ast}{\longrightarrow}   H^1_{DR}(\widehat E_-/A, (p)),
$$
defined by
$$
  \sum_{n\geq 1}\frac{c(n-1)}n \xi^n \mapsto  \sum_{n\geq 1} \frac{c(n-1)^\sigma}n  \Phi(\xi)^n =   \sum_{n\geq 1} \frac{c(n-1)^\sigma}n  \Delta_-^n\xi^{pn} \mapsto \sum_{n\geq 1} \frac{c(n-1)}n \Delta_-^n\Delta_+^n\xi^{p^2n}.
$$

Let $\omega_-:=\sum_n \frac{a(n-1)}n \xi^n$  represent the class of the standard invariant holomorphic differential one form on $E_-/A$ .  Then $\widetilde{\Phi}^\ast(\omega_-)$ belongs to the subspace of the holomorphic 1-forms in $  H^1_{DR}(\widehat E_+/A, (p))$  generated by  $\omega_+:=\sum_n \frac{a(n-1)^\sigma}n \xi^n$. In particular, for some $\varepsilon \in \{\pm 1\}$,
 $$
     \widetilde{\Phi}^\ast(\omega_-)= \frac{\varepsilon p}{u_p} \omega_+ = \frac{\varepsilon p}{u_p}\sum_n \frac{a(n-1)^\sigma}n \xi^n.
$$
Hence, by comparing the $p$-th power coefficients, we get
  $$
    a(p-1)^\sigma \equiv   \varepsilon   \Delta_-  u_p  \pmod {pA}.  
  $$
Further by \cite{LongRamakrishna, WIN2}, one has
\begin{equation} \label{eq:v+eval}
 \left( [F_{0,{\frac 1d,\frac{d-1}d}}]_{p-1}(v_-) \right)^\sigma= [F_{0,{\frac 1d,\frac{d-1}d}}]_{p-1}(v_+) \equiv    \varepsilon
 \Delta_-  u_p \pmod {p^2 A}.
\end{equation}
Similarly, since 
$\widetilde{\Phi}^\ast(\omega_+) = \frac{\varepsilon p}{u_p}\omega_-$, 
the maps
$$
H^1_{DR}(\widehat E_+/A, (p)) \overset{\widetilde{\Phi}^\ast}{\longrightarrow} H^1_{DR}(\widehat E_-/A, (p))\overset{\widetilde{\Phi}^\ast}{\longrightarrow} H^1_{DR}(\widehat E_+/A, (p))
$$ 
give that 
\begin{equation} \label{eq:v-eval}
[F_{0,{\frac 1d,\frac{d-1}d}}]_{p-1}(v_-) \equiv \varepsilon \Delta_+  u_p \pmod {p^2 A}.
\end{equation}
\color{black}
Combining \eqref{eq:v+eval} and \eqref{eq:v-eval}, we obtain 
\[
[F_{0,{\frac 1d,\frac{d-1}d}}]_{p-1}(v_+)   [F_{0,{\frac 1d,\frac{d-1}d}}]_{p-1}(v_-)  \equiv u_p^2 \pmod{p^2 A}.
\]

Since $u_p^2\in \Z_p$, and using Lemma \ref{lemma: p-Clausen}, it follows that 
$$[F_{0,{\frac 1d,\frac{d-1}d}}]_{p-1}(v_+)[F_{0,{\frac 1d,\frac{d-1}d}}]_{p-1}(v_-) \in \Z_p[\l]=\Z_p.$$  
Thus, 
\begin{equation}\label{eq:up2}
[F_{0,{\frac 1d,\frac{d-1}d}}]_{p-1}(v_+)[F_{0,{\frac 1d,\frac{d-1}d}}]_{p-1}(v_-) \equiv u_p^2 \pmod {p^2}.
\end{equation}
Finally,  from Lemma \ref{lemma: p-Clausen} and \eqref{eq:up2}, 
\[
[F_{0,(\frac12,\frac{1}{d},\frac{d-1}{d})} ]_{p-1} (\l) \equiv  \left(\frac{k_d}p\right)\left[F_{0,(\frac{1}{d},\frac{d-1}{d})}\right]_{p-1}\left(v_-\right)\cdot \left[F_{0,(\frac{1}{d},\frac{d-1}{d})}\right]_{p-1} \left(v_+\right) \equiv  \left(\frac{k_d}p\right)    u_p^2   \pmod {p^2}, 
\]
which completes the proof. 
\end{proof}

\begin{remark}
\label{rem: action}
    We note that a second common eigenvector $\nu$ of the action of $\End(E_-)$ on $ H^1_{DR}(\widehat E_-/A,(p))$ corresponds to the class of a differential of the second kind, and  similarly, $\widetilde{\Phi}^\ast(\nu)=\varepsilon u_p\nu$ for $\varepsilon \in \{\pm 1\}$.  One can then use similar steps as in the proof above to show the claims in Proposition \ref{lem:main}.
\end{remark}

We note that the following arises as a corollary of Lemma \ref{lemma: p-Clausen}.

\begin{corollary}
Fix a prime $p\geq 5$. As a polynomial in $t$, 
\begin{multline*}
    [F_{0,{(\frac 1d,\frac{d-1}d})}]_{p-1}\left(\frac{1+ \sqrt{1-t}}{2}\right)  - \left(\frac{k_d}p\right) [F_{0,{(\frac 1d,\frac{d-1}d})}]_{p-1}\left(\frac{1- \sqrt{1-t}}{2}\right)  \\
    \in 
    \begin{cases}
       \sqrt{1-t} \cdot p^2\Z_p[t], &\mbox{ if } \left(\frac{k_d}p\right)=1,\\
       p^2\Z_p[t], &\mbox{ if } \left(\frac{k_d}p\right)=-1.
    \end{cases}
\end{multline*}
In particular, for $1\neq t\in \Z_p$, in the case when $\left(\frac{k_d}p\right)=1$, 
$$
  \sum_{k=1}^{p-1} \frac{\left(\frac 1d\right)_k \left(\frac {d-1}d\right)_k}{k!^2} \left(\frac 12\right)^{\!k} \, \sum_{j=0}^{\lfloor{\frac{k-1}{2}\rfloor}s} \binom{k}{2j+1}(1-t)^j \equiv 0 \pmod{p^2},
$$
and when  $\left(\frac{k_d}p\right)=-1$, 
$$
  1+\sum_{k=1}^{p-1} \frac{\left(\frac 1d\right)_k \left(\frac {d-1}d\right)_k}{k!^2} \left(\frac 12\right)^{\!k} \, \sum_{j=0}^{\lfloor{\frac k2}\rfloor} \binom{k}{2j}(1-t)^j \equiv 0 \pmod{p^2}. 
$$
\end{corollary}

\begin{proof}
Set $s=\sqrt{1-t}$ and $u=\frac{1+ s}{2}$ so that $1-u=\frac{1- s}{2}$. From the second statement of Lemma \ref{lemma: p-Clausen}, we have
\begin{align*}
   g(s):= [F_{0,{(\frac 1d,\frac{d-1}d})}]_{p-1}\left(u\right)  - \left(\frac{k_d}p\right) [F_{0,{(\frac 1d,\frac{d-1}d})}]_{p-1}\left(1-u\right)  \in p^2\Z_p[u]= p^2\Z_p[s]. 
\end{align*}
On the other hand, expanding $g(s)$ as a series in $s$ yields
\begin{align*}
g(s) &= \sum_{k=0}^{p-1} \frac{\left(\frac 1d\right)_k \left(\frac {d-1}d\right)_k}{k!^2}\left( u^k-\left(\frac{k_d}p\right)(1-u)^k \right) \\
&= \sum_{k=0}^{p-1} \frac{\left(\frac 1d\right)_k \left(\frac {d-1}d\right)_k}{k!^2} \left(\frac 12\right)^k \sum_{j=0}^k\left((1+s)^k-\left(\frac{k_d}p\right)(1-s)^k\right) \\
&= \sum_{k=0}^{p-1} \frac{\left(\frac 1d\right)_k \left(\frac {d-1}d\right)_k}{k!^2} \left(\frac 12\right)^k \sum_{j=0}^k\binom kj s^j\left(1-(-1)^j\left(\frac{k_d}p\right)\right).
\end{align*}

Observe that 
$$
 \sum_{j=0}^k\binom kj s^j\left(1-(-1)^j\left(\frac{k_d}p\right)\right)= 
\begin{cases} 
2\displaystyle \sum_{\substack{j=0 \\ j  \mbox{ odd} }}^k\binom kj s^j=2s\displaystyle \sum_{j=0}^{\lfloor{\frac{k-1}{2}\rfloor}} \binom{k}{2j+1}s^{2j} & \text{ if } \left(\frac{k_d}p\right)=1, \\
2\displaystyle \sum_{\substack{j=0 \\ j  \mbox{ even} }}^k\binom kj s^j=2\displaystyle \sum_{j=0}^{\lfloor{\frac{k}{2}\rfloor}} \binom{k}{2j}s^{2j} & \text{ if } \left(\frac{k_d}p\right)=-1.
\end{cases}
$$
Hence, 
\[
g(s)=
\begin{cases}
2\sqrt{1-t}\displaystyle \sum_{k=1}^{p-1} \frac{\left(\frac 1d\right)_k \left(\frac {d-1}d\right)_k}{k!^2} \left(\frac 12\right)^k \displaystyle \sum_{j=0}^{\lfloor{\frac{k-1}{2}\rfloor}} \binom{k}{2j+1}(1-t)^{j} \in  \sqrt{1-t} \cdot p^2\Z_p[t] &\text{ if } \left(\frac{k_d}p\right)=1, \\
2\displaystyle \sum_{k=1}^{p-1} \frac{\left(\frac 1d\right)_k \left(\frac {d-1}d\right)_k}{k!^2} \left(\frac 12\right)^k \displaystyle \sum_{j=0}^{\lfloor{\frac{k}{2}\rfloor}} \binom{k}{2j}(1-t)^{j} \in p^2\Z_p[t] &\text{ if } \left(\frac{k_d}p\right)=-1,
\end{cases}
\]
and the result follows.
\end{proof}

\subsection{A quick survey on corresponding hypergeometric character sums and supercongrunces}\label{galreps}

In this section, we summarize the relationship between hypergeometric character sums and truncated hypergeometric functions, and the modularity of hypergeometric Galois representations (see
\cite[et al.]{LLT2, HLLT, RRV} for example) attached to the data provided in this paper.  We will use the $H_q$-character sums described in the work of Beukers, Cohen, and Mellit \cite{BCM}.   

Let $\F_q$ be a finite field of odd characteristic and use $ \widehat{\F_q^\times}$ to denote the group of multiplicative characters of $\F_q^\times$.   For a given pair of multisets  $\alpha=\{a_1,\cdots,a_n\}$ and $\beta=\{1,b_2,\cdots,b_n\}$  with entries in $\Q^\times$, we denote  the positive least  common denominator of all $a_i,b_j$ by $M:=\rm{lcd}(\alpha\cup \beta)$. For a finite field $\F_q$, 
we fix $\omega$  a generator of $\widehat{\F_q^\times}$.  Following \cite{BCM}, for $q\equiv 1\pmod M$  and any $t\in \F_q$, define 
$$
H_q(\alpha,\beta;t;\omega):=\frac {1 }{1-q} \sum_{k=0}^{q-2} \omega^k((-1)^n t)
  \prod_{j=1}^n \frac{g(\omega^{(q-1)a_j+ k})}{g(\omega^{(q-1)a_j})}
  \frac{g(\ol\omega^{(q-1)b_j+ k)})}{g(\ol\omega^{(q-1)b_j})}, 
		$$  where $\ol \omega=\omega^{-1}$ and $g(\chi):=\sum_{x\in\F_q}\Psi(x)\chi(x)$ is the Gauss sum of $\chi$ with respect to a fixed choice of nontrivial additive character $\Psi$ of $\F_q$.

We say that the datum $HD = \{\alpha, \beta\}$ is defined over $\Q$ if both $\prod_{j=1}^n(X-e^{2\pi i a_j})$ and $\prod_{j=1}^n(X-e^{2\pi i b_j})$ are in $\Z[X]$. In such a special situation,  
one can rewrite $H_q(\alpha,\beta;t;\omega)$ by using the reflection and multiplication formulas of Gauss sums  \cite[Theorem 1.3]{BCM} and drop the condition $q\equiv 1\pmod M$.  Based on results of Katz and Beukers-Cohen-Mellit, there exist explicit algebraic varieties attached to such data and their corresponding Galois representations can be determined by $H_p(HD;t)$.  For simplicity, we state these results for $\beta=\{1,\ldots, 1\}$ below. 

\begin{theorem}[See \cite{Katz, Katz09, BCM, LLT2}] 
Let  $\ell$ be a prime. Given a datum   $HD:=\{\alpha=\{a_1,\cdots,a_n\},\,  \beta=\{1,\ldots, 1\}\}$ with $a_i\neq 1$, $M = \rm{lcd}(\alpha)$, for any  $\l \in \Z[\zeta_M,1/M] \backslash \{0\}$, there exists an $\ell$-adic Galois representation $$\rho_{HD,\ell}: G(M):=\mbox{Gal}(\ol \Q/\Q(\zeta_M))\rightarrow GL(W_{\l}),$$ unramified almost everywhere, such that at each prime ideal $\wp$ of  $ \Z[\zeta_M,1/(M\ell \l)]$ with norm $N(\wp)$,
\begin{equation*} 
\mbox{Tr} \rho_{HD,\ell}(\text{Frob}_\wp)= H_{N(\wp)}(\alpha,\beta; 1/\l;\omega_\wp),  
\end{equation*} 
where $\text{Frob}_\wp$ stands for the  geometric  Frobenius conjugacy class of $G(M)$ at $\wp$. Moreover,
\begin{enumerate}
\item When $\l\neq 1$, $dim_{\overline \Q_\ell}W_{\l} = n$ and all roots of the characteristic polynomial of $\rho_{HD,\ell}(\Frob_\wp)$  are algebraic with the same absolute value $N(\wp)^{(n-1)/2}$ under all Archimedean embeddings.
\item  When $\l=1$, $dim_{\overline \Q_\ell}W_{\l} = n-1$. 
\end{enumerate}
Furthermore, if $HD$ is defined over $\Q$ and $\l \in \Q$, then there exists an $\ell$-adic representation $\rho^{BCM}_{HD,\ell}$ of $G_\Q:=\rm{Gal}(\ol\Q/\Q)$ such that 
$$ \rho^{BCM}_{HD,\ell}\mid_{G(M)}\simeq \rho_{HD,\ell} $$ and for each prime $p\nmid M\ell$, $$ \mbox{Tr} \rho^{BCM}_{HD,\ell}(\text{Frob}_p)= H_{p}(\alpha,\beta; 1/\l).$$
\end{theorem}

We now summarize the relations among the hypergeometric functions in different settings based on works including \cite{AOP, BCM, WIN2, HLLT, Mortenson} and Proposition~\ref{sec: WIN2proof} for the specific data in this paper.  Letting $\phi_q$ be the quadratic character on $\F_q$, we rewrite the relevant $H_q$ as
\begin{align*}
 H_q&\left( \left\{\frac 12, \frac 12 \right\},\{1,1\};t;\omega\right):=\frac {\phi_q(-1) }{q(1-q)} \sum_{k=0}^{q-2} 
  g(\phi_q \omega^{k})^2g(\ol\omega^k)^2\omega^k(t), \\
   H_q&\left( \left\{\frac 12, \frac 12, \frac12 \right\},\{1,1,1\};t;\omega\right):=\frac {g(\phi_q) }{q^2(q-1)} \sum_{k=0}^{q-2} 
  g(\phi_q \omega^{k})^3g(\ol\omega^k)^3\omega^k(-t),
\end{align*}
and for $d \in \{ 3, 4, 6\}$, 
\begin{align*}
   H_q\left( \left\{\frac 1d, \frac {d-1}d \right\},\{1,1\};t;\omega\right)&:=\frac {1 }{1-q} \sum_{k=0}^{q-2} g(\ol\omega^k)^2S_d(\omega^k)\omega^k(t), \\
     H_q\left( \left\{\frac 12, \frac 1d, \frac {d-1}d \right\},\{1, 1,1\};t;\omega\right)&:=\frac {\phi_q(-1)g(\phi_q) }{q(q-1)}  \sum_{k=0}^{q-2}  g(\phi \omega^{k})g(\ol\omega^k)^3S_d(\omega^k)\omega^k(-t),
\end{align*}
where $$S_d(\chi):=\displaystyle \prod_{n\mid d}\left(\frac{g(\chi^n)}{g(\chi)}\,\chi(n^{-n})\right)^{\mu(d/n)}$$ and $\mu(\cdot)$ is the M\"obius $\mu$-function.  Precisely, we obtain the values in Table \ref{tab:S_d}.

\begingroup
\begin{table}[h!] \label{tab:S_d}
\renewcommand{\arraystretch}{1.9}
\begin{tabular}{|c||c|c|c|}\hline
$d$ & $3$ & $4$ & $6$ \\ \hline
$S_d(\chi)$ & $\frac{g(\chi^3)}{g(\chi)}\chi(3^{-3})$ &  $\frac{g(\chi^4)}{g(\chi^2)}\chi(2^{-6})$ & $\frac{g(\chi)g(\chi^6)}{g(\chi^2)g(\chi^3)}\chi(2^{-4}3^{-3})$ \\ \hline
\end{tabular}
\medskip
\caption{$S_d(\chi)$ for $d\in \{3,4,6\}$}
\end{table}
\endgroup

\begin{theorem}(\cite{BCM, HLLT, Zudilin-Modularity} and Proposition \ref{prop: traces}) \label{thm: Hq-traces}
For $q=p^k$ with $p\geq 5$ prime, let $t\in \F_q$ and $t\neq 0, 1$ such that $ \widetilde{E}_{d}(t)$ in Table \ref{fams} is an elliptic curve defined over $\F_q$. Then 
$$
    \# \widetilde{E}_{d}(t)(\F_q)=q+1-H_q\left( \left\{d, \frac {d-1}d \right\},\{1,1\}; t\right).
$$
Furthermore, 
\begin{align*}
     H_q\left( \left\{d,  \frac {d-1}d\right\},\{1,1\}; t\right) & = \phi_q(k_d)H_q\left( \left\{d, \frac {d-1}d \right\},\{1,1\}; 1-t\right), \\
     H_p\left(\left\{d,  \frac {d-1}d \right\},\{1,1\}; 1\right) & =  \left(\frac {k_d}p\right),
\end{align*}
    where $\phi_q$ is the quadratic character of $\F_q$, and $k_2=k_6=-1$, $k_3=-3$, $k_4=-2$. 
\end{theorem}

Building on the work of Evans-Greene \cite{Evans-Greene} and Theorem \ref{thm: Hq-traces}, we have the following $H_q$-version of Clausen's formula. The Galois interpretation can be found in \cite{LLT2}. 

\begin{theorem}[$H_p$-Clausen formula \cite{Evans-Greene, HLLT}]
Let $p>3$ be a prime and $t\in \F_p\setminus\{0,1\}$. Let $\sqrt {1-t}$ be a fixed element of $\F_{p^2}$ such that $\sqrt {1-t}^2 =1-t$. Then 
$$
  \left(\frac{1-t}p\right) H_p  \left( \left\{  \frac 12,\frac12, \frac 12 \right\},\{1,1,1\};t\right) =  H_p \left( \left\{ \frac14, \frac 34 \right\},\{1,1\};\frac t{t-1}\right)^2-p.  
$$
For $d=2$, $3$, $4$, $6$, we have 
    \begin{align*}
         H_p & \left( \left \{  \frac 12,\frac1d, 1-\frac 1d \right \},\{1,1,1\};t\right) \\
         &= \begin{cases}
              H_p\left ( \left\{\frac1d, 1-\frac 1d \right \},\{1,1\};\frac{1-\sqrt{1-t}}2\right)^2-p, & \quad \mbox{ if $1-t$ is a square in  }  \F_p,\\
            \left(\frac{\kappa_d}p\right) H_{p^2}\left (\{\frac1d, 1-\frac 1d\},\{1,1\};\frac{1-\sqrt{1-t}}2\right)-p,& \quad \mbox{ if $1-t$ is not  a square in  }  \F_p, 
          \end{cases}  \\ 
          H_p & \left(  \left\{  \frac 12,\frac1d, 1-\frac 1d \right\},\{1,1,1\};1\right) =  H_p\left( \left\{\frac1d, 1-\frac 1d \right\},\{1,1\};\frac12\right)^2-\left(1+\left(\frac{k_d}p\right)\right)p,
    \end{align*}

 where  $k_2=k_6=-1$, $k_3=-3$, $k_4=-2$. 
\end{theorem}

Combining the previous information gives the following. 

\begin{theorem}\label{thm: Hp and F super}
Assume the notation and assumptions in Theorem \ref{thm:WIN2a} and that $\l \in \Q$ is a $CM$-value.  Further let $K_{CM}$ be the discriminant of the corresponding CM field. Then for primes $p\geq 7$ such that $\l\in \Z_p^\times$, 
\begin{align*}
   H_p\left( \left\{\frac 12,\frac1d, 1-\frac 1d \right\},\{1,1,1\};\l\right)&- \left(\frac{K_{CM}}p\right) \left(\frac{1-\l}p\right)p  \\
   &=\begin{cases}
0 & \text{if }  \widetilde{E}_d(u) \text{ is supersingular at } p\\
u_p^2 + p^2/u_p^2 & \text{if }  \widetilde{E}_d(u) \text{ is ordinary at } p \text{ and }\left(\frac{1-\l}p \right)=1, \\ 
\left( \frac{k_d}{p} \right)\left(u_p^2 + p^2/u_p^2\right)  & \text{if }\widetilde{E}_d(u) \text{ is ordinary at } p \text{ and }\left(\frac{1-\l}p \right)=-1, 
\end{cases}
\end{align*} 
and 
$$
     [F_{0,(\frac12, \frac{1}{d}, \frac{d-1}{d})}]_{p-1} (\l)\equiv   H_p\left( \left\{\frac 12,\frac1d, 1-\frac 1d \right\},\{1,1,1\};\l\right)- \left(\frac{K_{CM}}p\right) \left(\frac{1-\l}p\right)p  \pmod{p^2},
$$
where 
 $u=\frac{1-\sqrt{1-\l}}2$, $u_p$ is a fixed unit root of the geometric Frobenius at $p$ acting on the first cohomology of the elliptic curve $\widetilde{E}_d(u)$, and $k_2=k_6=-1$, $k_3=-3$, $k_4=-2$.  
\end{theorem}

\begin{remark}\label{rmk: Hp and F}
In general, one has
$$ [F_{0,(\frac{1}{d}, \frac{d-1}{d})}]_{p-1} (t)\equiv   H_p\left( \left\{\frac1d, 1-\frac 1d \right\},\{1,1\};t\right)\quad \pmod {p}, $$
$$ [F_{0,(\frac12, \frac{1}{d}, \frac{d-1}{d})}]_{p-1} (t)\equiv   H_p\left( \left\{\frac 12,\frac1d, 1-\frac 1d \right\},\{1,1,1\};t\right)\quad \pmod {p}, $$
for $t \in \Q^\times$ with $t\in \Z_p^\times$, see \cite{Long18, RRV} for example.  
\end{remark}

\begin{corollary}
\label{Cor: Hpweight3}
Assume the notation and assumptions as above. Then, for  primes $p\nmid 2d$ such that $\l\in \Z_p^\times$, there exists a weight-3 Hecke eigenform $f$ depending on $d$ and $\l$ such that 
\[
H_p\left( \left\{\frac 12,\frac1d, 1-\frac 1d \right\},\{1,1,1\};\l\right) - \left(\frac{K_{CM}}p\right) \left(\frac{1-\l}p\right)p  = a_p(f),
\] 
where $a_p(f)$ is the $p$-th Fourier coefficient of $f$.  
\end{corollary} 
The modularity of the corresponding Galois representations in this corollary is essentially coming from the modularity of CM hypergeometric elliptic curves (or K3 surfaces \cite{HGK3}) in the background.  The modularity of the hypergeometric representations arising from datum defined over $\Q$ has been developed by numerous researchers  (see \cite{Ahlgren-Ono-CalabiYau, Barman-Saikia, Fuselier-McCarthy, LLT2, LTYZ, McCarthy-Papanikolas,  Tripathi-Meher,   Mortenson-padic,  RRV,  Salerno}).  In recent work, Allen, Grove, Long, and the fifth author \cite{HMM1} establish the modularity of a certain type of representations, those that can be lifted to $\mbox{Gal}(\ol \Q/\Q)$, from the commutative formal group law and an explicit construction of the corresponding cusp forms via the modular forms on arithmetic triangle groups as discussed in this paper and for example in \cite{Beukers-supercongruence, WIN5, HLLT}.


\begin{thebibliography}{10}

\bibitem{Ahlgren-Ono-CalabiYau}
Scott Ahlgren and Ken Ono.
\newblock Modularity of a certain {C}alabi-{Y}au threefold.
\newblock {\em Monatsh. Math.}, 129(3):177--190, 2000.

\bibitem{AOP}
Scott Ahlgren, Ken Ono, and David Penniston.
\newblock Zeta functions of an infinite family of {$K3$} surfaces.
\newblock {\em Amer. J. Math.}, 124(2):353--368, 2002.

\bibitem{HMM2}
Michael Allen, Brian Grove, Ling Long, and Fang-Ting Tu.
\newblock The {E}xplicit {H}ypergeometric-{M}odularity {M}ethod {II}, 2024
  (preprint).

\bibitem{HMM1}
Michael Allen, Brian Grove, Ling Long, and Fang-Ting Tu.
\newblock The {E}xplicit {H}ypergeometric-{M}odularity {M}ethod {I},
  arXiv:2404.00711 (2024).

\bibitem{AAR}
George~E Andrews, Richard Askey, Ranjan Roy, Ranjan Roy, and Richard Askey.
\newblock {\em Special functions}, volume~71.
\newblock Cambridge university press Cambridge, 1999.

\bibitem{WIN5}
Angelica Babei, Lea Beneish, Manami Roy, Holly Swisher, Bella Tobin, and
  Fang-Ting Tu.
\newblock Generalized {R}amanujan-{S}ato series arising from modular forms.
\newblock In {\em Research directions in number theory---{W}omen in {N}umbers
  {V}}, volume~33 of {\em Assoc. Women Math. Ser.}, pages 87--131. Springer,
  Cham, 2024.

\bibitem{Barman-Saikia}
Rupam Barman and Neelam Saikia.
\newblock Certain character sums and hypergeometric series.
\newblock {\em Pacific J. Math.}, 295(2):271--289, 2018.

\bibitem{Bauer}
Gustav Bauer.
\newblock Von den {C}oefficienten der {R}eihen von {K}ugelfunctionen einer
  {V}ariablen.
\newblock {\em J. Reine Angew. Math.}, 56:101--121, 1859.

\bibitem{Beukers-supercongruence}
Frits Beukers.
\newblock Supercongruences using modular forms, arXiv: 2403.03301 (2024).

\bibitem{BCM}
Frits Beukers, Henri Cohen, and Anton Mellit.
\newblock Finite hypergeometric functions.
\newblock {\em Pure Appl. Math. Q.}, 11(4):559--589, 2015.

\bibitem{BB}
Jonathan~M. Borwein and Peter~B. Borwein.
\newblock {\em Pi and the {AGM}}.
\newblock Canadian Mathematical Society Series of Monographs and Advanced
  Texts. John Wiley \& Sons Inc., New York, 1987.
\newblock A study in analytic number theory and computational complexity, A
  Wiley-Interscience Publication.

\bibitem{BGHZ}
Jan~Hendrik Bruinier, Gerard van~der Geer, G\"{u}nter Harder, and Don Zagier.
\newblock {\em The 1-2-3 of modular forms}.
\newblock Universitext. Springer-Verlag, Berlin, 2008.
\newblock Lectures from the Summer School on Modular Forms and their
  Applications held in Nordfjordeid, June 2004, Edited by Kristian Ranestad.

\bibitem{CCL}
Heng~Huat Chan, Song~Heng Chan, and Zhiguo Liu.
\newblock Domb's numbers and {R}amanujan-{S}ato type series for {$1/\pi$}.
\newblock {\em Adv. Math.}, 186(2):396--410, 2004.

\bibitem{ChenGlebov}
Imin Chen and Gleb Glebov.
\newblock On {C}hudnovsky--{R}amanujan type formulae.
\newblock {\em The Ramanujan Journal}, 46(3):677--712, 2018.

\bibitem{ChenGlebovGoenka}
Imin Chen, Gleb Glebov, and Ritesh Goenka.
\newblock Chudnovsky-{R}amanujan type formulae for non-compact arithmetic
  triangle groups.
\newblock {\em J. Number Theory}, 241:603--654, 2022.

\bibitem{WIN2}
Sarah Chisholm, Alyson Deines, Ling Long, Gabriele Nebe, and Holly Swisher.
\newblock $p$--adic analogues of {R}amanujan type formulas for 1/$\pi$.
\newblock {\em Mathematics}, 1(1):9--30, 2013.

\bibitem{CC}
David~V. Chudnovsky and Gregory~V. Chudnovsky.
\newblock Approximations and complex multiplication according to {R}amanujan.
\newblock In {\em Ramanujan revisited ({U}rbana-{C}hampaign, {I}ll., 1987)},
  pages 375--472. Academic Press, Boston, MA, 1988.

\bibitem{Cohen-Stromberg}
Henri Cohen and Fredrik Str\"{o}mberg.
\newblock {\em Modular forms}, volume 179 of {\em Graduate Studies in
  Mathematics}.
\newblock American Mathematical Society, Providence, RI, 2017.
\newblock A classical approach.

\bibitem{Coster-Van}
M.~J. Coster and L.~Van~Hamme.
\newblock Supercongruences of {A}tkin and {S}winnerton-{D}yer type for
  {L}egendre polynomials.
\newblock {\em J. Number Theory}, 38(3):265--286, 1991.

\bibitem{Evans-Greene}
Ron Evans and John Greene.
\newblock Clausen's theorem and hypergeometric functions over finite fields.
\newblock {\em Finite Fields Appl.}, 15(1):97--109, 2009.

\bibitem{FLS}
Nuno Freitas, Bao~V Le~Hung, and Samir Siksek.
\newblock Elliptic curves over real quadratic fields are modular.
\newblock {\em Inventiones mathematicae}, 201(1):159--206, 2015.

\bibitem{WIN3+}
Jenny Fuselier, Ling Long, Ravi Ramakrishna, Holly Swisher, and Fang-Ting Tu.
\newblock Hypergeometric functions over finite fields.
\newblock {\em Mem. Amer. Math. Soc.}, 280(1382):vii+124, 2022.

\bibitem{Fuselier-McCarthy}
Jenny~G. Fuselier and Dermot McCarthy.
\newblock Hypergeometric type identities in the {$p$}-adic setting and modular
  forms.
\newblock {\em Proc. Amer. Math. Soc.}, 144(4):1493--1508, 2016.

\bibitem{Guillera-Zudilin}
Jes\'{u}s Guillera and Wadim Zudilin.
\newblock ``{D}ivergent'' {R}amanujan-type supercongruences.
\newblock {\em Proc. Amer. Math. Soc.}, 140(3):765--777, 2012.

\bibitem{GuoZudilin}
Victor~JW Guo and Wadim Zudilin.
\newblock Ramanujan-type formulae for 1/$\pi$: q-analogues.
\newblock {\em Integral Transforms and Special Functions}, 29(7):505--513,
  2018.

\bibitem{GCloud}
Emma Haruka~Iwao.
\newblock Even more pi in the sky: Calculating 100 trillion digits of pi on
  google cloud, 2022.
\newblock
  https://cloud.google.com/blog/products/compute/calculating-100-trillion-digits-of-pi-on-google-cloud.

\bibitem{HLLT}
Jerome~W. Hoffman, Wen-Ching~Winnie Li, Ling Long, and Fang-Ting Tu.
\newblock Traces of {H}ecke {O}perators via {H}ypergeometric {C}haracter
  {S}ums, 2024z; arXiv: 2408.02918.

\bibitem{Katz70}
Nicholas~M. Katz.
\newblock The regularity theorem in algebraic geometry.
\newblock In {\em Actes du {C}ongr\`es {I}nternational des {M}ath\'{e}maticiens
  ({N}ice, 1970), {T}ome 1,}, pages 437--443. ,, 1971.

\bibitem{Katz72}
Nicholas~M. Katz.
\newblock Algebraic solutions of differential equations ({$p$}-curvature and
  the {H}odge filtration).
\newblock {\em Invent. Math.}, 18:1--118, 1972.

\bibitem{Katz}
Nicholas~M. Katz.
\newblock {$p$}-adic properties of modular schemes and modular forms.
\newblock In {\em Modular functions of one variable, {III} ({P}roc. {I}nternat.
  {S}ummer {S}chool, {U}niv. {A}ntwerp, {A}ntwerp, 1972)}, volume Vol. 350 of
  {\em Lecture Notes in Math.}, pages 69--190. Springer, Berlin-New York, 1973.

\bibitem{Katz-crystalline}
Nicholas~M. Katz.
\newblock Crystalline cohomology, {D}ieudonn\'e{} modules, and {J}acobi sums.
\newblock In {\em Automorphic forms, representation theory and arithmetic
  ({B}ombay, 1979)}, volume~10 of {\em Tata Inst. Fundam. Res. Stud. Math.},
  pages 165--246. Springer, Berlin-New York, 1981.

\bibitem{Katz09}
Nicholas~M. Katz.
\newblock Another look at the {D}work family.
\newblock In {\em Algebra, arithmetic, and geometry: in honor of {Y}u. {I}.
  {M}anin. {V}ol. {II}}, volume 270 of {\em Progr. Math.}, pages 89--126.
  Birkh\"{a}user Boston, Boston, MA, 2009.

\bibitem{HGK3}
A.~Klemm, W.~Lerche, and P.~Mayr.
\newblock {$K_3$}-fibrations and heterotic--type {II} string duality.
\newblock {\em Phys. Lett. B}, 357(3):313--322, 1995.

\bibitem{Lang_EF}
Serge Lang.
\newblock Elliptic functions.
\newblock In {\em Elliptic Functions}, pages 5--21. Springer, 1987.

\bibitem{LLT2}
Wen-Ching~W. {Li}, Ling {Long}, and Fang-Ting {Tu}.
\newblock {A Whipple $_7F_6$ formula revisted}.
\newblock {\em La Matematica}, 1-51, 2022.

\bibitem{LL}
Wen-Ching~Winnie Li and Ling Long.
\newblock Atkin and {S}winnerton-{D}yer congruences and noncongruence modular
  forms.
\newblock In {\em Algebraic number theory and related topics 2012}, volume B51
  of {\em RIMS K\^oky\^uroku Bessatsu}, pages 269--299. Res. Inst. Math. Sci.
  (RIMS), Kyoto, 2014.

\bibitem{LLT}
Wen-Ching~Winnie {Li}, Ling {Long}, and Fang-Ting {Tu}.
\newblock {Computing special L-values of certain modular forms with complex
  multiplication}.
\newblock {\em SIGMA 14 (2018), 090}, August 2018.

\bibitem{LMFDB}
The {LMFDB Collaboration}.
\newblock The {L}-functions and modular forms database.
\newblock \url{http://www.lmfdb.org}, 2021.
\newblock [Online; accessed 25 August 2021].

\bibitem{Long18}
Ling Long.
\newblock Some numeric hypergeometric supercongruences.
\newblock In {\em Vertex operator algebras, number theory and related topics},
  volume 753 of {\em Contemp. Math.}, pages 139--156. Amer. Math. Soc.,
  Providence, RI, 2020.

\bibitem{LongRamakrishna}
Ling Long and Ravi Ramakrishna.
\newblock Some supercongruences occurring in truncated hypergeometric series.
\newblock {\em Advances in Mathematics}, 290:773--808, 2016.

\bibitem{LTYZ}
Ling Long, Fang-Ting Tu, Noriko Yui, and Wadim Zudilin.
\newblock Supercongruences for rigid hypergeometric {C}alabi-{Y}au threefolds.
\newblock {\em Adv. Math.}, 393:Paper No. 108058, 49, 2021.

\bibitem{Manin}
Yuri~Ivanovich Manin.
\newblock Algebraic curves over fields with differentiation.
\newblock {\em Izvestiya Rossiiskoi Akademii Nauk. Seriya Matematicheskaya},
  22(6):737--756, 1958.

\bibitem{Marcus}
Daniel~A. Marcus.
\newblock {\em Number fields}.
\newblock Universitext. Springer, Cham, 2018.
\newblock Second edition of [ MR0457396], With a foreword by Barry Mazur.

\bibitem{McCarthy-Papanikolas}
Dermot McCarthy and Matthew~A. Papanikolas.
\newblock A finite field hypergeometric function associated to eigenvalues of a
  {S}iegel eigenform.
\newblock {\em Int. J. Number Theory}, 11(8):2431--2450, 2015.

\bibitem{Mortenson}
Eric Mortenson.
\newblock Supercongruences for truncated {$_{n+1}\!F_n$} hypergeometric series
  with applications to certain weight three newforms.
\newblock {\em Proc. Amer. Math. Soc.}, 133(2):321--330, 2005.

\bibitem{Mortenson-padic}
Eric Mortenson.
\newblock A {$p$}-adic supercongruence conjecture of van {H}amme.
\newblock {\em Proc. Amer. Math. Soc.}, 136(12):4321--4328, 2008.

\bibitem{Ogg-book}
Andrew Ogg.
\newblock {\em Modular forms and {D}irichlet series}.
\newblock W. A. Benjamin, Inc., New York-Amsterdam, 1969.

\bibitem{Osburn-Straub}
Robert Osburn and Armin Straub.
\newblock Interpolated sequences and critical {$L$}-values of modular forms.
\newblock In {\em Elliptic integrals, elliptic functions and modular forms in
  quantum field theory}, Texts Monogr. Symbol. Comput., pages 327--349.
  Springer, Cham, 2019.

\bibitem{OsburnZudilin}
Robert Osburn and Wadim Zudilin.
\newblock On the ({K}.2) supercongruence of {V}an {H}amme.
\newblock {\em J. Math. Anal. Appl.}, 433(1):706--711, 2016.

\bibitem{Ramanujan}
Srinivasa Ramanujan.
\newblock Modular equations and approximations to {$\pi$} [{Q}uart. {J}.
  {M}ath. {\bf 45} (1914), 350--372].
\newblock In {\em Collected papers of {S}rinivasa {R}amanujan}, pages 23--39.
  AMS Chelsea Publ., Providence, RI, 2000.

\bibitem{RRV}
David~P. Roberts and Fernando Rodriguez~Villegas.
\newblock Hypergeometric motives.
\newblock {\em Notices Amer. Math. Soc.}, 69(6):914--929, 2022.

\bibitem{sagemath}
Inc. SageMath.
\newblock {\em CoCalc Collaborative Calculation Online}, 2016.
\newblock {\tt https://cocalc.com/}.

\bibitem{Salerno}
Adriana Salerno.
\newblock Counting points over finite fields and hypergeometric functions.
\newblock {\em Funct. Approx. Comment. Math.}, 49(1):137--157, 2013.

\bibitem{Sato}
T~Sato.
\newblock Ap{\'e}ry numbers and ramanujan’s series for 1/$\pi$.
\newblock In {\em Abstract of a talk presented at the Annual meeting of the
  Mathematical Society of Japan}, pages 28--31, 2002.

\bibitem{ChowlaSelberg}
Atle Selberg and S.~Chowla.
\newblock On {E}pstein's zeta-function.
\newblock {\em J. Reine Angew. Math.}, 227:86--110, 1967.

\bibitem{Silverman-adv}
Joseph~H. Silverman.
\newblock {\em Advanced topics in the arithmetic of elliptic curves}, volume
  151 of {\em Graduate Texts in Mathematics}.
\newblock Springer-Verlag, New York, 1994.

\bibitem{Silverman}
Joseph~H Silverman.
\newblock {\em The arithmetic of elliptic curves}, volume 106.
\newblock Springer Science \& Business Media, 2009.

\bibitem{Beukers-Stienstra}
Jan Stienstra and Frits Beukers.
\newblock On the {P}icard-{F}uchs equation and the formal {B}rauer group of
  certain elliptic {$K3$}-surfaces.
\newblock {\em Math. Ann.}, 271(2):269--304, 1985.

\bibitem{Swisher}
Holly Swisher.
\newblock On the supercongruence conjectures of van {H}amme.
\newblock {\em Research in the Mathematical Sciences}, 2(1):1--21, 2015.

\bibitem{Tripathi-Meher}
Mohit Tripathi and Jaban Meher.
\newblock {$_4F_3$}-{G}aussian hypergeometric series and traces of {F}robenius
  for elliptic curves.
\newblock {\em Res. Math. Sci.}, 9(4):Paper No. 63, 24, 2022.

\bibitem{vanHamme}
L.~van Hamme.
\newblock Some conjectures concerning partial sums of generalized
  hypergeometric series.
\newblock In {\em {$p$}-adic functional analysis ({N}ijmegen, 1996)}, volume
  192 of {\em Lecture Notes in Pure and Appl. Math.}, pages 223--236. Dekker,
  New York, 1997.

\bibitem{Weil}
Andr\'{e} Weil.
\newblock \"{U}ber die {B}estimmung {D}irichletscher {R}eihen durch
  {F}unktionalgleichungen.
\newblock {\em Math. Ann.}, 168:149--156, 1967.

\bibitem{Yang-Schwarzian}
Yifan Yang.
\newblock Schwarzian differential equations and {H}ecke eigenforms on {S}himura
  curves.
\newblock {\em Compositio Mathematica}, 149(1):1--31, 2013.

\bibitem{Yoshida-DE}
Masaaki Yoshida.
\newblock {\em Fuchsian differential equations}, volume E11 of {\em Aspects of
  Mathematics}.
\newblock Friedr. Vieweg \& Sohn, Braunschweig, 1987.
\newblock With special emphasis on the Gauss-Schwarz theory.

\bibitem{Zagier}
Don Zagier.
\newblock Traces of singular moduli.
\newblock In {\em Motives, polylogarithms and {H}odge theory, {P}art {I}
  ({I}rvine, {CA}, 1998)}, volume~3 of {\em Int. Press Lect. Ser.}, pages
  211--244. Int. Press, Somerville, MA, 2002.

\bibitem{Zudilin}
Wadim Zudilin.
\newblock Ramanujan-type formulae for {$1/\pi$}: a second wind?
\newblock In {\em Modular forms and string duality}, volume~54 of {\em Fields
  Inst. Commun.}, pages 179--188. Amer. Math. Soc., Providence, RI, 2008.

\bibitem{Zudilin-supercongruence}
Wadim Zudilin.
\newblock Ramanujan-type supercongruences.
\newblock {\em J. Number Theory}, 129(8):1848--1857, 2009.

\bibitem{Zudilin-Modularity}
Wadim Zudilin.
\newblock A hypergeometric version of the modularity of rigid {C}alabi-{Y}au
  manifolds.
\newblock {\em SIGMA Symmetry Integrability Geom. Methods Appl.}, 14:Paper No.
  086, 16, 2018.

\end{thebibliography}
\end{document}